\documentclass{amsart}
\usepackage{amsmath,amsthm,amsfonts,amssymb,latexsym,mathrsfs,graphicx}
\usepackage{pgf,tikz}
\usepackage{mathrsfs}
\usetikzlibrary{arrows}

\usepackage{hyperref}

\usepackage{enumerate}
\usepackage[shortlabels]{enumitem}
\usepackage{color}

\usepackage{pstricks-add}
\usepackage{graphicx}
\usepackage{pgf,tikz,pgfplots}
\usepackage{mathrsfs}
\usetikzlibrary{arrows}
\definecolor{ududff}{rgb}{0.30196078431372547,0.30196078431372547,1}

\newtheorem{theorem}{Theorem}[section]
\newtheorem*{Main}{Main Theorem}

\newtheorem{lemma}[theorem]{Lemma}
\newtheorem{proposition}[theorem]{Proposition}
\theoremstyle{definition}

\newtheorem{rem}[theorem]{Remark}

\newenvironment{enumeratei}{\begin{enumerate}[\upshape (a)]}
    {\end{enumerate}}

\def\irr#1{{\rm Irr}(#1)}
\def\cent#1#2{{\bf C}_{#1}(#2)}

\def\syl#1#2{{\rm Syl}_#1(#2)}

\def\oh#1#2{{\bf O}_{#1}(#2)}
\def\zent#1{{\bf Z}(#1)}
\def\V#1{{V}(#1)}

\newcommand{\N}{{\mathbb N}}

\newcommand{\F}{{\mathbb F}}
\newcommand{\K}{{\mathbb K}}

\def\irr#1{{\rm Irr}(#1)}

\def\cent#1#2{{\bf C}_{#1}(#2)}

\def\syl#1#2{{\rm Syl}_#1(#2)}

\def\norm#1#2{{\bf N}_{#1}(#2)}
\def\oh#1#2{{\bf O}_{#1}(#2)}

\def\zent#1{{\bf Z}(#1)}

\def \nq{\mathcal{N}_q}

\def\SL#1{{\rm SL}_{2}(#1)}
\def\PSL#1{{\rm PSL}_{2}(#1)}

\def\V#1{{\rm V}(#1)}

\def\C{{\Bbb C}}

\def\irr#1{{\rm Irr}(#1)}

\def\cd#1{{\rm cd}(#1)}

\def\cent#1#2{{\bf C}_{#1}(#2)}
\def\syl#1#2{{\rm Syl}_#1(#2)}
\def\oh#1#2{{\bf O}_{#1}(#2)}

\def\zent#1{{\bf Z}(#1)}

\def\ker#1{{\rm ker}(#1)}
\def\norm#1#2{{\bf N}_{#1}(#2)}

\mathchardef\coso="2023

\def \nq{\mathcal{N}_q}

\begin{document}

\title[Degree graphs with a cut-vertex. II]{Non-solvable groups whose character degree graph has a cut-vertex. II}

\author[]{Silvio Dolfi}
\address{Silvio Dolfi, Dipartimento di Matematica e Informatica U. Dini,\newline
Universit\`a degli Studi di Firenze, viale Morgagni 67/a,
50134 Firenze, Italy.}
\email{silvio.dolfi@unifi.it}

\author[]{Emanuele Pacifici}
\address{Emanuele Pacifici, Dipartimento di Matematica e Informatica U. Dini,\newline
Universit\`a degli Studi di Firenze, viale Morgagni 67/a,
50134 Firenze, Italy.}
\email{emanuele.pacifici@unifi.it}

\author[]{Lucia Sanus}
\address{Lucia Sanus, Departament de Matem\`atiques, Facultat de
 Matem\`atiques, \newline
Universitat de Val\`encia,
46100 Burjassot, Val\`encia, Spain.}
\email{lucia.sanus@uv.es}

\thanks{The first two authors are partially supported by the italian INdAM-GNSAGA. The research of the third author is partially supported by Ministerio de Ciencia e Innovaci\'on PID2019-103854GB-I00. This research has been carried out during a visit of the third author at the Dipartimento di Matematica e Informatica ``Ulisse Dini” (DIMAI) of Universit\`a degli Studi di Firenze. She thanks DIMAI for the financial support and hospitality.}


\begin{abstract} 
Let $G$ be a finite group, and let \(\cd G\) denote the set of degrees of the irreducible complex characters of $G$. Define then the \emph{character degree graph} \(\Delta(G)\)  as the (simple undirected) graph whose vertices are the prime divisors of the numbers in $\cd{G}$, and two distinct vertices $p$, $q$ are adjacent if and only if $pq$ divides some number in $\cd{G}$. This paper continues the work, started in \cite{DPSS}, toward the classification of the finite non-solvable groups whose degree graph possesses a \emph{cut-vertex}, i.e. a vertex whose removal increases the number of connected components of the graph. While, in \cite{DPSS}, groups with no composition factors isomorphic to \(\PSL{t^a}\) (for any prime power \(t^a\geq 4\)) were treated, here we consider the complementary situation \emph{in the case when \(t\) is odd and \(t^a> 5\)}.  The proof of this classification will be then completed in the third and last paper of this series (\cite{DPS3}), that deals with the case $t=2$.
\end{abstract}

\maketitle

\section{Introduction}

Let $G$ be a finite group, and let \({\rm cd}(G)\) denote the \emph{degree set} of $G$, i.e. the set of degrees of the irreducible complex characters of $G$. In the paper ``Research in Representation Theory at Mainz (1984--1990)" (\cite{mainz}), Bertram Huppert writes: ``Several special results, known since some time, made it clear that the structure of a finite group $G$ is controlled to a large extent by the type of the prime-number-decomposition of the degrees of the irreducible characters of $G$ over $\mathbb{C}$". That work highly contributed to boost the interest of many authors, and the study of the arithmetical properties of the degree set, both on their own account and in connection with the structure of the group, is nowadays a well-established and classical research topic in the representation theory of finite groups.

Part of the discussion in \cite{mainz} concerns the \emph{character degree graph} \(\Delta(G)\) of a finite group $G$ (\emph{degree graph} for short), a tool that has been devised in order to investigate the arithmetical structure of the degree set \({\rm cd}(G)\). This is the simple undirected graph whose vertex set \(\V G\) consists of the primes dividing the numbers in $\cd{G}$, and such that two distinct vertices $p$ and $q$ are adjacent if and only if $pq$ divides some number in $\cd{G}$. As discussed, for instance, in the survey \cite{overview}, many results in the literature illustrate the deep link between graph-theoretical properties of \(\Delta(G)\) (in particular, connectivity properties) and the group structure of \(G\).

The series of three papers including the present one (together with \cite{DPSS,DPS3}) is a contribution in this framework. Namely, our purpose is to characterize the finite non-solvable groups whose degree graph has a \emph{cut-vertex}, i.e. a vertex whose removal increases the number of connected components of the graph. We mention that an analysis of the solvable case is carried out in \cite{LM}. 

The main results of the whole series are Theorem~A, Theorem~B and Theorem~C of \cite{DPSS}: they are stated in full details and commented in \cite[Section~2]{DPSS}, where also the relevant graphs are described with some figures (we refer the reader to \cite[Section~2]{DPSS}, as well as to the Introduction of \cite{DPSS} for a more exhaustive presentation of the problem). In particular Theorem~C of \cite{DPSS}, which provides a characterization of the finite non-solvable groups whose degree graph has a cut-vertex \emph{and it is disconnected}, is entirely proved in that paper and will not be discussed further. As regards Theorem~A and Theorem~B of \cite{DPSS}, they deal with finite non-solvable groups whose degree graph has \emph{connectivity degree~$1$} (i.e. it has a cut-vertex \emph{and it is connected}), and they are only partially proved in \cite{DPSS}. As we said, we do not reproduce the full statements of these theorems here but, for the convenience of the reader, we summarize next some properties of the relevant class of groups.

If $G$ is a finite non-solvable group whose degree graph has connectivity degree $1$, then $G$ has a unique non-solvable composition factor $S$, belonging to a short list of isomorphism types:
${\rm{PSL}}_2(t^a)$, ${\rm{Sz}}(2^a)$, ${\rm{PSL}}_3(4)$,   ${\rm{M}}_{11}$, ${\rm{J}}_1$; moreover, if  
$S \not\cong {\rm{PSL}}_2(t^a)$, then $G$ has a (minimal) normal subgroup isomorphic to $S$. 
Finally, denoting by $R$ the solvable radical of $G$, the vertex set of $\Delta(G)$ consists of the primes in \(\pi(G/R)\) (the set of prime divisors of the
order of the almost-simple group $G/R$) and, if not already there, the cut-vertex of $\Delta(G)$. 
We also remark that in all cases, except  possibly when the non-abelian  simple section $S$ is isomorphic to the Janko group ${\rm{J}}_1$,
the cut-vertex $p$ of $\Delta(G)$ is a complete vertex (i.e. it is adjacent to all other vertices) of $\Delta(G)$; moreover, the graph obtained from $\Delta(G)$ by removing the vertex $p$ (and all the edges incident to $p$) has exactly two
connected components, which are complete graphs, and one of them consists of a single vertex. 

While in \cite{DPSS} we classify the finite non-solvable groups $G$ such that \(\Delta(G)\) has connectivity degree $1$ \emph{and whose unique non-solvable composition factor $S$ is not isomorphic to \(\PSL{t^a}\)} for any prime power \(t^a\geq 4\) (which covers cases~(a)--(d) of \cite[Theorem~A]{DPSS}), here we consider the situation when $S\cong\PSL{t^a}$ \emph{for an odd prime \(t\) with \(t^a>5\)}, thus covering case~(e) of \cite[Theorem~A]{DPSS}. Note that in conclusion (b) of the following theorem, by the ``natural module" for \(K/L\cong\SL{t^a}\) we mean the standard \(2\)-dimensional module for $K/L$ over the field with \(t^a\) elements, or any of its Galois conjugates, seen as a \(2a\)-dimensional $K/L$-module over the field with \(t\) elements.

\begin{Main}
\label{main}
Let \(G\) be a finite group and let \(R\) be its solvable radical. Assume that \(G\) has a composition factor \(S\cong \PSL{t^a}\), for a suitable odd prime \(t\) with \(t^a> 5\), and let \(p\) be a prime number. Then, denoting by $K$ the last term in the derived series of $G$, the graph \(\Delta(G)\) is connected and has cut-vertex \(p\) if and only if: \(G/R\) is an almost-simple group with socle isomorphic to $S$, $\V G=\pi(G/R)\cup\{p\}$, the order of \(G/KR\) is not a multiple of \(t\), the prime $p$ is not $t$, and one of the following holds.
\begin{enumeratei} 
\item \(K\) is isomorphic to \(\PSL{t^a}\) or to \(\SL{t^a}\), and $\V{G/K}=\{p\}$.
\item \(K\) contains a minimal normal subgroup \(L\) of \(G\) such that  \(K/L\) is isomorphic to \(\SL{t^a}\) and \(L\) is the natural module for \(K/L\); moreover, $\V{G/K}=\{p\}$.
 \item \(t^a=13\) and $p=2$. \(K\) contains a minimal normal subgroup \(L\) of \(G\) such that \(K/L\) is isomorphic to \(\SL{13}\), and \(L\) is one of the two \(6\)-dimensional irreducible modules for \(\SL{13}\) over the field with three elements. Moreover, $\V{G/K}\subseteq\{2\}$.
\end{enumeratei}
In all cases, \(p\) is a complete vertex and the unique cut-vertex of \(\Delta(G)\), and it is the unique neighbour of \(t\) in \(\Delta(G)\) with the only exception of case {\rm(c)}.
\end{Main}

As already mentioned, the even-characteristic case is treated in \cite{DPS3}, and it covers the remaining case (f) of \cite[Theorem~A]{DPSS} together with \cite[Theorem~B]{DPSS}. As regards the ``only if" part of the above theorem, the main difference from the situation treated in \cite{DPSS} is that here the subgroup \(K\) need not be minimal normal in $G$. In fact, \(K\) can be isomorphic to \(\SL{t^a}\), or even, it is possible to have a non-trivial normal subgroup \(L\) of \(G\) such that \(K/L\cong\SL{t^a}\). Much of the work carried out in this paper consists in showing that such a subgroup $L$ is minimal normal in $G$, and in controlling the (conjugation) action of $K/L$ on~$L$. To this end, two results concerning orbit properties in certain actions of \(\SL{t^a}\) (Theorem~\ref{TipoIeIIPieni}, that should be compared with Lemma~3.10 of \cite{DPSS}, and Theorem~\ref{TipoIeII}) turn out to be crucial in our analysis; these might be of interest on their own.

To conclude, we point out that the structure of the groups appearing in the Main Theorem (as well as of the corresponding graphs), quite surprisingly, does not fall too far from the structure of the finite non-solvable groups whose degree graph has two connected components (see Theorem~\ref{LewisWhite}). We will comment on this fact in Remark~\ref{comparison}. At any rate, the graphs related to the groups of the Main Theorem are displayed in Table~\ref{c}.
In the following discussion, every group is tacitly assumed to be a finite group.

\begin{figure}[h]\label{c}
\begin{tikzpicture}[line cap=round,line join=round,>=triangle 45,x=1.0cm,y=1.0cm]
\clip(4.957179639559777,2.3750082393717458) rectangle (18.521254061592398,11.439779926291106);
\draw (9.1,5.7) node[anchor=north west] {\scriptsize $3$};
\draw (10.184310758294298,5.7) node[anchor=north west] {\scriptsize $2$};
\draw (11.132693197558368,6.5) node[anchor=north west] {\scriptsize $13$};
\draw (11.19885941425121,4.5) node[anchor=north west] {\scriptsize $7$};
\draw [line width=0.8pt] (5.254421391528034,3.2671357237409713)-- (17.92714866425532,3.236226632831879);
\draw (10.376232345281109,11.28) node[anchor=north west] {\scriptsize TABLE 1 };
\draw [line width=0.8pt] (5.198498650078525,10.801501349921475)-- (17.871225922805806,10.770592259012384);
\draw [line width=0.8pt] (7.,9.) circle (1.0053670359753404cm);
\draw [line width=0.8pt] (8.331283305519394,9.029609897345264) circle (1.0053670359753388cm);
\draw (5.934891812366582,10.403175864769914) node[anchor=north west] {\scriptsize (clique)};
\draw (8.089047229687633,10.425231270334194) node[anchor=north west] {\scriptsize (clique)};
\draw (7.449440468323494,9.146017747605915) node[anchor=north west] {\scriptsize $\pi_0$};
\draw (7.603828307273459,9.543015047762967) node[anchor=north west] {\scriptsize $p$};
\draw [line width=0.8pt] (7.603698516681713,9.33599150858786)-- (7.628081865416518,10.272038341433083);
\draw (7.471495873887775,10.7) node[anchor=north west] {\scriptsize $t$};
\draw (6.523113434623705,9.410682614377283) node[anchor=north west] {\scriptsize $\pi_1$};
\draw (8.463989124280406,9.410682614377283) node[anchor=north west] {\scriptsize $\pi_2$};
\draw (5.159949917773811,8.1) node[anchor=north west] {\scriptsize $\pi_1\cup\pi_0=\pi(t^a-1)\cup\pi(G/KR)\cup\{p\}$};
\draw (5.159949917773811,7.8) node[anchor=north west] {\scriptsize $\pi_2\cup\pi_0=\pi(t^a+1)\cup\pi(G/KR)\cup\{p\}$};
\draw (5.890781001238021,7.38158530246346) node[anchor=north west] {\footnotesize $\text{Case (a) of Main Theorem }$};
\draw [line width=0.8pt] (14.,9.) circle (1.0053670359753388cm);
\draw (13.580843215193525,10.401397487027037) node[anchor=north west] {\scriptsize (clique)};
\draw (11.98,9.5) node[anchor=north west] {\scriptsize $2$};
\draw (13.1,9.2) node[anchor=north west] {\scriptsize $p$};
\draw [line width=0.8pt] (12.1650717681562,8.99869517803798)-- (12.99463296402466,9.);
\draw (11.881075636822439,7.38158530246346) node[anchor=north west] {\footnotesize $\text{Case (b) of Main Theorem }$};
\draw [line width=0.8pt] (9.324548106534905,5.239031874151431)-- (10.368257126347096,5.239031874151431);
\draw [line width=0.8pt] (10.368257126347096,5.229339251136105)-- (11.375302658765392,5.992094751913781);
\draw [line width=0.8pt] (11.375302658765392,5.992094751913781)-- (11.375302658765392,4.4908258236212195);
\draw [line width=0.8pt] (10.368257126347096,5.229339251136105)-- (11.375302658765392,4.4908258236212195);
\draw (8.683041830001739,4.071773117899884) node[anchor=north west] {\footnotesize $\text{Case (c) of Main Theorem} $};
\begin{scriptsize}
\draw [fill=black] (7.603698516681713,9.33599150858786) circle (1.5pt);
\draw [fill=black] (7.628081865416518,10.272038341433083) circle (1.5pt);
\draw [fill=black] (12.1650717681562,8.99869517803798) circle (1.5pt);
\draw [fill=black] (12.99463296402466,9.) circle (1.5pt);
\draw [fill=black] (9.324548106534905,5.249031874151431) circle (1.5pt);
\draw [fill=black] (10.368257126347096,5.229339251136105) circle (1.5pt);
\draw [fill=black] (11.375302658765392,5.992094751913781) circle (1.5pt);
\draw [fill=black] (11.375302658765392,4.4908258236212195) circle (1.5pt);
\end{scriptsize}
\end{tikzpicture}
\end{figure}

\section{Preliminaries}

Given a group \(G\), we denote by \(\Delta(G)\) the degree graph of \(G\) as defined in the Introduction. Our notation concerning character theory is standard, and we will freely use basic facts and concepts such as Ito-Michler's theorem, Clifford's theory, Gallagher's theorem, character triples and results about extension of characters (see \cite{Is}). 

As customary (and already used), given a positive integer \(n\), the set of prime divisors of \(n\) will be denoted by \(\pi(n)\), but we simply write \(\pi(G)\) for \(\pi(|G|)\). Moreover, for a prime power \(q\), we use the symbol \(\mathbb{F}_q\) to denote the field having \(q\) elements.

\smallskip
We start by recalling some structural properties of the \(2\)-dimensional special linear or projective special linear groups. Although this paper treats groups of this kind in odd characteristic, most of the results in this section and in the following one will be stated and proved with no restrictions about the characteristic.

\begin{rem}\label{Subgroups}
Recall that, for an odd prime \(t\), the group \(\PSL{t^a}\) has order \(\dfrac{t^a(t^a-1)(t^a+1)}{2}\). Moreover, as stated in \cite[II.8.27]{Hu}, the proper subgroups of this group are of the following types.
\begin{itemize}
\item[(i$_+$)] Dihedral groups of order $t^a + 1$ and their subgroups. 
\item[(i$_-$)] Dihedral groups of order $t^a - 1$ and their subgroups. 
\item[(ii)] Frobenius groups with elementary abelian kernel of order $t^a$ and cyclic complements of order $(t^a -1)/2$, and their subgroups;
\item[(iii)] $A_4$, $S_4$ 
    or $A_5$; 
    \item[(iv)] $\PSL{t^b}$ or ${\rm PGL}_2(t^b)$, where $b$ divides $a$. 
\end{itemize}  
\end{rem}

In our discussion, we will freely refer to the above labels when dealing with a subgroup of \(\PSL{t^a}\). By a subgroup of type (i) we will mean a subgroup that is either of type (i$_-$) or of type (i$_+$).

\begin{lemma}
\label{PSL2}
Let \(G\cong\SL{t^a}\) or \(G\cong\PSL{t^a}\), where \(t\) is a prime and \(t^a\geq 4\). Let \(r\) be an odd prime divisor of \(t^a-1\), and let \(R\) be a subgroup of \(G\) with \(|R|=r^b\) for a suitable \(b\in\N-\{0\}\). Then \(R\) lies in the normalizer in \(G\) of precisely two Sylow \(t\)-subgroups of \(G\).
\end{lemma}

\begin{proof}
We start by observing that the number of Sylow \(t\)-subgroups of \(G\) is \(t^a+1\) and, for \(T\in{\rm{Syl}}_t(G)\), the number of subgroups of order \(r^b\) lying in \(\norm G T\) is \(t^a\). Moreover, the total number of subgroups of \(G\) having order \(r^b\) is \(t^a(t^a+1)/2\). 

Now, consider the set \[X=\{(R_0,T)\; \mid\; T\in\syl t G,\, |R_0|=r^b,\, R_0\subseteq\norm G T\}.\]
On one hand we get \[|X|=\sum_{T\in\syl t G}t^a=t^a(t^a+1);\]
on the other hand, if \(n\) denotes the number of Sylow \(t\)-subgroups of \(G\) that are normalized by a given subgroup of order \(r^b\), we also have \[|X|=\sum_{|R_0|=r^b}n=t^a(t^a+1)/2\cdot n.\]
The desired conclusion is then readily achieved.
\end{proof}

The following results are more specific of the context we will analyze.

\begin{lemma}[\mbox{\cite[Theorem~5.2]{W}}]
  \label{PSL2bis}
Let $S \cong \PSL{t^a}$ or $S \cong \SL{t^a}$, with $t$ prime and $a \geq 1$. 
Let $\rho_{+} = \pi(t^a+1)$ and $\rho_{-} = \pi(t^a-1)$. For a subset $\rho$ of vertices 
of $\Delta(S)$, we denote by $\Delta_{\rho}$ the subgraph of $\Delta = \Delta(S)$ induced
by the subset $\rho$.
Then 
\begin{enumeratei}
\item if $t=2$ and $a \geq 2$, then $\Delta(S)$ has three connected components, $\{t\}$, $\Delta_{\rho_{+}}$ and 
$\Delta_{\rho_{-}}$, and each of them is a complete graph.  
\item if $t > 2$ and $t^a > 5$, then  $\Delta(S)$ has two connected components, 
$\{t\}$ and  $\Delta_{\rho_{+} \cup \rho_{-}}$; moreover, both  $\Delta_{\rho_{+}}$ and $\Delta_{\rho_{-}}$ are
complete graphs, no vertex  in $\rho_{+} - \{ 2 \}$ is adjacent to any vertex  in  
$\rho_{-} - \{ 2\}$ and $2$ is a complete vertex of $\Delta_{\rho_{+} \cup \rho_{-}}$. 
\end{enumeratei}
\end{lemma}

\begin{theorem}[\mbox{\cite[Theorem~3.9]{DPSS}}]\label{MoretoTiep}
  Let $G$ be an almost-simple group with socle $S$, and let $\delta = \pi(G) - \pi(S)$.
  If $\delta \neq \emptyset$, then $S$ is a simple group of Lie type, and every vertex in \(\delta\) is adjacent to every other vertex of $\Delta(G)$ that is not the characteristic of $S$.
\end{theorem}

Given a group $G$, we will denote by $R = R(G)$ the \emph{solvable radical} (i.e. the largest solvable normal subgroup), and
by $K = K(G)$ the \emph{solvable residual} (i.e. the smallest normal subgroup with a solvable factor group) of $G$. Equivalently, $K(G)$ is the last term of the derived series of $G$. 

\begin{lemma}\label{InfiniteCommutator}
Let \(G\) be a group and let \(R\) be its solvable radical. Assume that \(G/R\) is an almost-simple group with socle isomorphic to $\PSL{t^a}$, for a prime \(t\) with \(t^a> 4\) and \(t^a\neq 9\). Then, denoting by \(K\) the solvable residual of $G$, one of the following conclusions holds.

\begin{enumeratei} 
\item \(K\) is isomorphic to \(\PSL{t^a}\) or to \(\SL{t^a}\); 
\item \(K\) has a non-trivial normal subgroup \(L\) such that \(K/L\) is isomorphic to \(\PSL{t^a}\) or to \(\SL{t^a}\), and every non-principal irreducible character of \(L/L'\) is not invariant in \(K\).
\end{enumeratei}
\end{lemma}

\begin{proof}
Note that \(K\) is clearly non-trivial because \(G\) is non-solvable, so there exists a normal subgroup \(N\) of \(G\) such that \(K/N\) is a chief factor of \(G\). As \(K\) is perfect, it is easy to see that \(KR/NR\cong K/N\) is a non-solvable chief factor of \(G/R\); now, denoting by \(M/R\) the socle of the almost-simple group \(G/R\), we get \(NR=R\) (i.e. \(N=K\cap R\)) and \(KR=M\). Thus, by our hypothesis, \(K/N\cong M/R\) is isomorphic to \(\PSL{t^a}\) and we get conclusion (a) if \(N\) is trivial. 

Therefore we can assume \(N\neq 1\), and we can also assume that there exists a non-principal irreducibe character \(\mu\) of \(N/N'\) such that \(\mu\) is invariant in \(K\) (otherwise we get (b) setting \(L=N\)). 

So, let us define \(L\) to be \(\ker\mu\). We have that \(L\) is a normal subgroup of \(K\) and \(N/L\) is contained in \(\zent{K/L}\); thus, since \(K\) is perfect, we see that \(N/L\) embeds in the Schur multiplier of \(K/N\) (see \cite[Theorem~11.19]{Is}). Under our assumptions, this Schur multiplier is trivial if \(t=2\) and it has order \(2\) if \(t\neq 2\), so, in the present situation, we have \(t\neq 2\), \(|N/L|=2\) and \(K/L\cong\SL{t^a}\). Again we reach conclusion (a) if \(L\) is trivial. But if \(L\neq 1\), taking into accout that the Schur multiplier of \(K/L\) is trivial and arguing as above, then we see that \(L/L'\) does not have any non-principal irreducible character that is invariant in \(K\). We reached conclusion (b) in this case, and the proof is complete.
\end{proof}

Recall that, for $a$ and $n$ integers larger than $1$, a prime divisor $q$ of $a^n-1$ is called a \emph{primitive prime divisor} if $q$ does not divide $a^b -1$ for all $1 \leq b <n$. In this case, $n$ is the order of $a$ modulo $q$, so $n$ divides $q-1$. 
It is known (\cite[Theorem~6.2]{MW}) that $a^n - 1$ always has primitive prime divisors except when $n = 2$ and $a= 2^c -1$ for some integer $c$, or when $n=6$ and $a= 2$. 

\begin{lemma}\label{singer}
Let \(G\) be a group and let \(R\) be its solvable radical. Assume that \(G/R\) is an almost-simple group with socle isomorphic to $\PSL{t^a}$, for a prime \(t\) with \(t^a\geq 4\). Denoting by \(K\) the solvable residual of \(G\), assume that \(L\) is a minimal normal subgroup of \(G\), contained in \(K\), such that \(K/L\cong\SL{t^a}\) acts non-trivially on \(L\). 
Setting \(|L|=q^d\), where \(q\) is a suitable prime, let \(U\) be a Sylow \(u\)-subgroup of \(R\) for an odd prime \(u\) that does not divide \(q^d-1\). 
If there exists a primitive prime divisor \(p\) of \(q^d-1\) such that \(|K/L|\) is a multiple of \(p\), then \(U\subseteq\cent G L\).\end{lemma}

\begin{proof}
Set \(C=\cent G L\). Since \(L\) is an elementary abelian \(q\)-group of order \(q^{d}\), the factor group \(G/C\) embeds in \({\rm{GL}}_{d}(q)\). Denoting by \(S/C\) a subgroup of \(KC/C\) such that \(|S/C|=p\), by \cite[Theorem~2.11]{He} we see that \(S/C\) is contained in a Singer subgroup of \({\rm{GL}}_{d}(q)\): this is a cyclic group of order \(q^{d}-1\), which is a maximal cyclic subgroup of \({\rm{GL}}_{d}(q)\), and which acts fixed-point freely on \(L\). Now \(S/C\) acts irreducibly on \(L\), thus, by Schur's Lemma and by the fact that the ring of endomorphisms of \(L\) that commute with the action of \(S/C\) is a finite field, \(\cent{{\rm{GL}}_{d}(q)}{S/C}\) is a cyclic group and it is then forced to have order \(q^{d}-1\). 
On the other hand, setting \(N/L=\zent{K/L}\) and observing that \(N=K\cap R\) (so, \([K/N,R/N]=1\)), by coprimality we have \(K/N\cong(K/L)/(N/L)=\cent{K/L}{U}/(N/L)\), whence \(K/L=\cent{K/L}{U}\). In particular, as \(L\subseteq C\), the factor group \(UC/C\) centralizes \(KC/C\), which forces \(U\subseteq C\) by the discussion above. The proof is complete.
\end{proof}

The above result can be applied when \(L\) is isomorphic to the natural module for \(K/L\cong\SL{t^a}\) as far as \(t^{2a}-1\) has a primitive prime divisor. Nevertheless, we remark that the conclusion of Lemma~\ref{singer} is true in this situation even if \(t^{2a}-1\) does not have any primitive prime divisor, i.e. if \(a=1\) and \(t\) is a Mersenne prime, or if \(t^{2a}=2^6\): in the former case, in fact, the \(u\)-part of \(|KC/C|\) already exhausts the full \(u\)-part of \(|{\rm{GL}}_{2}(t)|\) because \(u\) does not divide \(|{\rm{GL}}_{2}(t):\SL{t}|=t-1\), therefore a Sylow \(u\)-subgroup of \(R\) is forced to be contained in~\(\cent G L\). In the latter case, we observe that \({\rm{GL}}_{6}(2)\) has a unique conjugacy class of elements of order \(2^3+1\); therefore, defining \(S/C\) to be a subgroup of order \(9\) of \(G/C\), the action of \(S/C\) on \(L\) is again fixed-point free and obviously irreducible, so the previous argument goes through.

For later use, we stress that Lemma~\ref{singer} also applies when \(K/L\cong\SL{13}\) and \(L\) is isomorphic to an irreducible \(K/L\)-module of dimension \(6\) over \(\F_3\), because \(7\) is a primitive prime divisor of \(3^6-1\) dividing \(|\SL{13}|\).

\begin{theorem}{\cite[Theorem~3.15]{DPSS}}\label{0.2}
Let \(G\) be a non-solvable group such that \(\Delta(G)\) is connected and it has a cut-vertex \(p\). Then, denoting by \(R\) the solvable radical of \(G\), we have that \(G/R\) is an almost-simple group and \(\V G=\pi(G/R)\cup\{p\}\). 
\end{theorem}

As an important reference for our discussion, we recall here the characterization of the finite non-solvable groups whose degree graph has two connected components, provided by M.L. Lewis and D.L. White in \cite{LW}. We will discuss the relationship between this and our main result in the last section of the present paper.

\begin{theorem}{\cite[Theorem~6.3]{LW}}
\label{LewisWhite}
Let \(G\) be a non-solvable group. Then \(\Delta(G)\) has two connected components if and only if there exist normal subgroups \(N\subseteq K\) of \(G\) such that, setting \(C/N=\cent{G/N}{K/N}\), the following conditions hold.
\begin{enumeratei}
\item \(K/N\simeq\PSL{t^a}\), where \(t\) is a prime with \(t^a\geq 4\).
\item \(G/K\) is abelian.
\item If \(t^a\neq 4\), then \(t\) does not divide \(|G/CK|\).
\item If \(N\neq 1\), then either \(K\cong\SL{t^a}\) or there exists a minimal normal subgroup \(L\) of \(G\) such that \(K/L\cong\SL{t^a}\) and \(L\) is isomorphic to the natural module for \(K/L\).
\item If \(t=2\) or \(t^a=5\), then either \(CK\neq G\) or \(N\neq 1\).
\item If \(t=2\) and \(K\) is an in {\rm{(d)}} in the case \(K\not\cong\SL{t^a}\), then every non-principal character in \(\irr L\) extends to its inertia subgroup in \(G\).
\end{enumeratei}
\end{theorem}

We recall that if $G$ satisfies the hypotheses of Theorem~\ref{LewisWhite}, then the solvable radical \(R\) of \(G\) coincides with \(C\) and \(\V G=\pi(G/R)\) (see \cite[Remark~3.8]{DPSS}). 

For later use, we also note that if a group \(K\) has a normal subgroup \(L\) such that \(K/L\cong\SL{t^a}\) for an odd prime $t$ (with $t^a\geq 5$) and \(L\) is isomorphic to the natural \(K/L\)-module, then by the previous theorem $\Delta(K)$ has two connected components. In particular, it is easy to check that \(\Delta(K)\) in this case has \(\{t\}\) as one connected component, whereas all the other vertices are pairwise adjacent.

\section{Some orbit theorems}

In this third preliminary section, we focus on some module actions of \(2\)-dimensional  special linear groups that will be crucial for our discussion. The main results of the section are Theorem~\ref{TipoIeIIPieni}, which deals with irreducible modules for \(\SL{t^a}\) in cross characteristic, and Theorem~\ref{TipoIeII} concerning actions on \(t\)-groups. We  will also prove another result on orbit sizes in this kind of linear actions (Lemma~\ref{brauer}), that will turn out to be useful. 

To begin with, we recall some special types of actions of groups on modules. 
Let \(H\) and \(V\) be groups, and assume that \(H\) acts by automorphisms on \(V\). Given a prime number  \(q\), we say that \emph{the pair \((H,V)\) satisfies the condition \(\nq\) } if  $q$ divides $|H: \cent HV|$ and, for every non-trivial \(v\in V\), there exists a Sylow \(q\)-subgroup \(Q\) of \(H\) such that \(Q\trianglelefteq \cent H v\) (see \cite{C}). If $(H, V)$ satisfies $\nq$ then \(V\) turns out to be an elementary abelian \(r\)-group for a suitable prime \(r\), and \(V\) is in fact an \emph{irreducible} module for \(H\) over the field $\mathbb{F}_r$ (see~Lemma~4 of~\cite{Z}). 

For \(H\cong\SL{t^a}\), we have the following result.

\begin{lemma}{\cite[Lemma~3.10]{DPSS}} \label{SL2Nq}
  Let $t$, $q$, $r$  be prime numbers, let $H = {\rm SL}_2(t^a)$ (with $t^a \geq 4$)  and  let  $V$ be an $H$-module over the field $\mathbb{F}_r$.
  Then $(H, V)$ satisfies $\nq $ if and only if 
 either $t^a = 5$ and $V$ is the natural module for $H/\cent HV \cong {\rm SL}_2(4)$ or $V$ is faithful and 
  one of the following holds.
  \begin{enumeratei}
  \item $t = q = r$ and $V$ is the natural $\mathbb{F}_r[H]$-module (so $|V| = t^{2a}$); 
  \item $q = r = 3$ and $(t^a, |V|) \in \{(5, 3^4), (13, 3^6)\}$. 
  \end{enumeratei}
\end{lemma}

Theorem~\ref{TipoIeIIPieni} is introduced by the following Lemma~\ref{brauer2}. Note that the numbers indicated in Table~\ref{Brauer3} of this lemma are the \emph{possible} dimensions of the relevant modules (only some of them actually show up).

\begin{lemma}\label{brauer2}
Let \(V\) be an irreducible module for \(G=\SL{t^a}\) over the field \(\F_q\), where \(t^a\geq 4\) and \(q\) is a prime in \(\pi(G)-\{t\}\). Also, let \(\F\) be a finite-degree field extension of \(\F_q\) such that \(\F\) is a splitting field for \(G\) and all its subgroups, and let \(\ell\) denote the composition length of the \(\F[G]\)-module \(V\otimes \F\). Then, for \(T\in\syl t G\), the maximal dimensions of \(\cent V T\) over \(\F_q\) (depending on \(\dim_{\F_q} V\)) are listed in Table~\ref{Brauer3}. 
\begin{table}[htbp]
\caption{\label{Brauer3} Maximal dimension of the centralizers of \(T\) in \(V\).}
\renewcommand{\arraystretch}{.95}
$$
\begin{array}{llllll} \hline\hline\\ 
\dim_{\F_q} V\quad\qquad\quad& {\ell\cdot t^a}\quad\quad & \ell\cdot(t^a+1) \quad\quad& \ell\cdot(t^a-1)\quad\quad& \ell\cdot\left(\dfrac{t^a+1}{2}\right)\quad\quad& \ell\cdot\left(\dfrac{t^a-1}{2}\right)\\ 
&&&& \text{\rm{(for }}t\neq 2)&  \text{\rm{(for }}t\neq 2) \\ \\
\hline
\\ 
\dim_{\F_q}\cent V T & \ell & 2\ell & 0 & \ell & 0 \\  \\    
\hline\hline
\end{array}
$$
\end{table}
\end{lemma}

\begin{proof}
We start by claiming that, for any field extension \(\K\) of \(\F_q\), we have \(\dim_{\K}\cent{V\otimes\K}{T}=\dim_{\F_q}\cent V T\). In fact, denote by \(U\) an \(\F_q[T]\)-submodule of \(V\) that is a direct complement for \(\cent V T\) in \(V\) (such a submodule \(U\) exists by Maschke's theorem); then, the \(\K[T]\)-module \(V\otimes \K\) decomposes as \((\cent T V\otimes \K)\oplus (U\otimes \K)\), and our claim easily follows from the fact  that the $1$-dimensional trivial \(\K[T]\)-module is not an irreducible constituent of \(U\otimes\K\) (this is ensured, for instance, by Lemma~1.12 in \cite[VII]{HB}). 

Now, \cite[VII, Theorem~2.6b)]{HB} yields the existence of a field \(\F\) as in our hypothesis. Since \(V\otimes\F\) is the direct sum of Galois conjugates of a suitable irreducible \(\F[G]\)-module \(W\), we get \(\dim_{\F}(V\otimes\F)=\ell\cdot \dim_{\F}W\) and \(\dim_{\F}\cent{V\otimes\F}{T}=\ell\cdot \dim_{\F}\cent W T\); therefore we will assume (absolute) irreducibility for the \(\F[G]\)-module \(V\otimes\F\) and the general bounds concerning the dimension of \(\cent V T\) over \(\F_q\) will follow at once. Given that, if we extend the field further to the algebraic closure \(\overline{\F}\) of \(\F\), then the module \(V\otimes{\overline{\F}}\) clearly remains irreducible.

Let \(\mathfrak{T}\) be an \(\overline{\F}\)-representation associated with \(V\otimes\overline{\F}\); by Maschke's theorem and by the fact that \(\overline{\F}\) is a splitting field for \(T\), the restriction \(\mathfrak{T}_T\) (up to equivalence) maps each element of \(T\) to a diagonal matrix, and \(\dim_{\overline{\F}}\cent{V\otimes\overline{\F}}T\) is the multiplicity of the $1$-dimensional trivial representation as a constituent of \(\mathfrak{T}_T\). Following the notation established in the paragraph preceding \cite[Lemma~15.1]{Is}, let \(R\) be the full ring of algebraic integers in the complex field, and let \(U\) be the subgroup of elements of order not divisible by \(q\) in the multiplicative group of \(\C\) (clearly, \(U\subseteq R\)); we can indentify \(\overline{\F}\) with the factor ring \(R/M\), where \(M\) is a maximal ideal of \(R\) containing \(qR\), and we consider the natural homomorphism \(*:R\rightarrow\overline{\F}\). By \cite[Lemma~15.1]{Is}, the restriction \(*_U\) maps \(U\) isomorphically onto the multiplicative group of \(\overline{\F}\). Now, consider the composite map \(\mathfrak{D}\) given by \(\mathfrak{T}_T\) followed by the map which applies the inverse of \(*_{U}\) to each entry of the relevant matrix: it is easy to see that \(\mathfrak{D}\) is a complex representation of \(T\) whose character \(\delta\) is the restriction to \(T\) of the Brauer character \(\theta\) afforded by \(\mathfrak{T}\), and what we want to determine is in fact the multiplicity \([1_T,\delta]\) of the trivial character of \(T\) as a constituent of \(\delta\). 

In view of the discussion of \cite[Sections 9.2--9.4]{Bo}, it turns out that \(\theta\) has a lift \(\chi\in\irr G\) (i.e. the restriction \(\widehat{\chi}\) of \(\chi\) to the set of \(q\)-regular elements of \(G\) coincides with \(\theta\)).
Now, as we have \(\delta=\chi_T\), we can refer to the ordinary character table of \(G\) (see for example \cite[Page 58]{Bo}) and compute \([1_T,\chi_T]\). The result of this computation can be found (setting \(\ell=1\), as we are assuming irreducibility for \(V\otimes\F\)) in the second row of Table~\ref{Brauer3}, where the maximal values for \(\dim_{\F_q}\cent V x\) are displayed according to the possible dimensions of \(V\) over \(\F_q\). 
\end{proof}

We are ready to prove the first main result of this section.

\begin{theorem} \label{TipoIeIIPieni} Let \(V\) be a non-trivial irreducible module for \(G=\SL{t^a}\) over the field \(\F_q\), where \(t^a\geq 4\) and \(q\neq t\) is a prime number. For odd primes \(r\in\pi(t^a-1)\) and \(s\in\pi(t^a+1)\) (possibly \(r=q\) or \(s=q\)) let \(R\), \(S\) be respectively a Sylow \(r\)-subgroup and a Sylow \(s\)-subgroup of \(G\), and let \(T\) be a Sylow \(t\)-subgroup of \(G\). Then, considering the sets 
\[V_{I_-}=\{v\in V\mid {\textnormal{ there exists }} z\in G {\textnormal{ such that }}R ^z\trianglelefteq\cent G v\},\] \[V_{I_+}=\{v\in V\mid {\textnormal{ there exists }} z\in G {\textnormal{ such that }}S ^z\trianglelefteq\cent G v\},\] \[V_{II}=\{v\in V\mid {\textnormal{ there exists }} z \in G {\textnormal{ such that }}T^z\trianglelefteq\cent G v\},\] we have that \(V-\{0\}\) strictly contains \(V_{I_-}\cup V_{II}\), \(V_{I_+}\cup V_{II}\), and \(V_{I_-}\cup V_{I_+}\), unless one of the following holds.
\begin{enumeratei}
\item \(G\cong\SL 5\), \(s=3\), \(|V|=3^4\) and \(V-\{0\}=V_{I_+}\),
\item \(G\cong\SL {13}\), \(r=3\), \(|V|=3^6\) and \(V-\{0\}=V_{I_-}\).
\end{enumeratei}
\end{theorem}

\begin{proof} Observe first that, if \(V-\{0\}\) is covered by just one of the sets \(V_{I_-}\), \(V_{I_+}\) or \(V_{II}\), then we get conclusions (a) or (b) by Lemma~\ref{SL2Nq}; therefore we will show that \(V-\{0\}\) cannot be covered as in the statement assuming that none of the relevant sets is empty. 

Given that, we start by ruling out the case when \(q\) is coprime to \(|G|\); in fact, in this situation, Theorem~2.3 of \cite{KP} ensures that there exists \(v\in V\) whose centralizer in \(G\) is trivial or of order \(2\), thus \(v\neq 0\) does not lie in any of the sets \(V_{I_+}\), \(V_{I_-}\) and \(V_{II}\).

We will then assume \(q\in\pi(G)-\{t\}\), and we will show first that \(V-\{0\}\) cannot be covered by \(V_{I_-}\) and \(V_{II}\). So, for a proof by contradiction, let us assume \(V-\{0\}=V_{I_-}\cup V_{II}\) (both the sets on the right-hand side being non-empty, as we said). 

If, for \(g\in G\), the element \(v\in V-\{0\}\) is centralized by both \(R\) and \(R^g\) with \(R^g\neq R\), then there exists a Sylow \(t\)-subgroup of \(G\) that is normalized by both \(R\) and \(R^g\). Since, by Lemma~\ref{PSL2}, \(R\) is contained in the normalizer of precisely two Sylow \(t\)-subgroups of \(G\), and since these normalizers contain \(t^a\) conjugates of \(R\) each, there are at most \(2(t^a-1)\) choices for \(R^g\). On the other hand, the total number of conjugates of \(R\) in \(G\) is \(\dfrac{t^{a}\cdot(t^a+1)}{2}\), so there certainly exists an element \(h\in G\) such that no element of \(V-\{0\}\) is centralized by both \(R\) and \(R^h\). As a consequence, we get \(\dim_{\F_q} V\geq 2\cdot \dim_{\F_q}\cent V R\). 

Now, the possible dimensions of \(V\) over \(\F_q\) are listed in Table~\ref{Brauer3}. Note that, \(V_{II}\) being non-empty, the dimension over \(\F_q\) of \(\cent V T\) cannot be \(0\), so the relevant dimensions in the present situation are \(\ell\cdot t^a\), \(\ell\cdot(t^a+1)\) and \(\ell\cdot\left(\dfrac{t^a+1}{2}\right)\). 

\medskip
Let us consider the case when \(\dim_{\F_q} V=\ell\cdot t^a\). Taking into account Table~\ref{Brauer3} together with the conclusions of the second-last paragraph above (and the fact that the number of Sylow \(t\)-subgroups of \(G\) is \(t^a+1\)), in order to get a contradiction it will be enough to show that the following inequality holds:\[q^{\ell\cdot t^a}-1>\dfrac{t^a\cdot(t^a+1)}{2}\cdot(q^{\ell\cdot t^{a}/2}-1)+(t^a+1)\cdot(q^{\ell}-1).\] Since \((t^a+1)\cdot(q^\ell-1)\) is clearly smaller than \(\dfrac{t^a\cdot(t^a+1)}{2}\cdot(q^{\ell\cdot t^{a}/2}-1)\), it is enough to analyze the inequality \[q^{\ell\cdot t^a}-1>t^a\cdot (t^{a}+1)\cdot (q^{\ell\cdot t^a/2}-1),\] which is in turn satisfied if \[q^{\ell\cdot t^a}>t^a\cdot (t^{a}+1)\cdot q^{\ell\cdot t^a/2}\] holds. The last inequality is obviously satisfied for every value of \(\ell\) provided it is satisfied for \(\ell=1\), and this happens whenever \(t^a\geq 17\); for smaller values of \(t^a\) we go back to the original inequality, and we see that it is not satisfied only by the following triples \((q,t^a,\ell)\): \((2,5,1)\), \((2,7,1)\), \((2,9,1)\), \((2,11,1)\), \((3,4,1)\). In any case, since the value of \(\ell\) is always \(1\), the \(\F_q[G]\)-module \(V\) is absolutely irreducible, and we are in a position to apply Corollary~1.2 of \cite{L}: there exists a regular orbit for the action of \(G\) on \(V\), a contradiction for us, \emph{except for the triple \((3,4,1)\)}. As regards \((3,4,1)\) (which is, by the way, the unique triple among those under consideration that corresponds to an existing module) a direct computation via GAP \cite{GAP} shows that \(\SL 4\) generates an orbit of size \(30\) in the action on its absoultely irreducible module of order \(3^4\). This contradiction completes the proof for the case \(\dim_{\F_q} V=\ell\cdot t^a\).

\medskip
We treat next the case \(\dim_{\F_q} V=\ell\cdot (t^a+1)\), so, we analyze the inequality
\[q^{\ell\cdot(t^a+1)}-1>\dfrac{t^a\cdot(t^a+1)}{2}\cdot(q^{\ell\cdot{\frac{t^a+1}{2}}}-1)+(t^a+1)\cdot(q^{2\ell}-1).\] As above, the second summand of the right-hand side is smaller than the first summand, so it is enough to have \[q^{\ell\cdot(t^a+1)}-1>t^a\cdot(t^a+1)\cdot(q^{\ell\cdot{\frac{t^a+1}{2}}}-1),\] which reduces to \[q^{\ell\cdot(t^a+1)}>t^a\cdot(t^a+1)\cdot q^{\ell\cdot{\frac{t^a+1}{2}}}.\] If the last inequality holds for \(\ell=1\), which happens as soon as \(t^a\geq 16\), then it clealry holds for every value of \(\ell\). On the other hand, for smaller values of \(t^a\), the original inequality is always satisfied except when the triple \((q,t^a,\ell)\) lies in \(\{(2,5,1),(2,7,1),(2,9,1),(2,11,1)\}\). Making use of \cite[Corollary~1.2]{L} as above, only the triple \((2,7,1)\) has to be checked: a computation via GAP \cite{GAP} shows that \(\SL 7\) generates an orbit of size \(21\) in the action on its absolutely irreducible module of order \(2^8\), the final contradiction for this case.


\medskip
To conclude, we consider the case \(\dim_{\F_q} V=\ell\cdot\left(\dfrac{t^a+1}{2}\right)\), only for \(t\neq 2\). Thus, we analyze the inequality
\[q^{\ell\cdot\frac{t^a+1}{2}}-1>\dfrac{t^a\cdot(t^a+1)}{2}\cdot(q^{\ell\cdot\frac{t^a+1}{4}}-1)+(t^a+1)\cdot(q^\ell-1).\] Similarly to the previous cases, we reduce to 
\[q^{\ell\cdot\frac{t^a+1}{2}}>t^a\cdot(t^a+1)\cdot q^{\ell\cdot\frac{t^a+1}{4}},\]
which is satisfied for every value of \(\ell\) as soon as it is satisfied for \(\ell=1\); this happens whenever \(t^a\geq 43\). Looking at smaller values in the original inequality for \(\ell=1\), and using Corollary~1.2 of \cite{L}, it turns out that only the triples \((3,11,1)\) and \((3,13,1)\) have to be checked. A direct computation via GAP \cite{GAP} shows that \(\SL{11}\) has two absolutely irreducible modules of order \(3^7\): in both of them, the three sets \(V_{I_-}\) (for \(r=5\)), \(V_{I_+}\) (for \(s=3\)) and \(V_{II}\) are all non-empty, so our assumptions are not satisfied. As regards the absolutely irreducible module of order \(3^7\) for  \(\SL {13}\), it has elements whose centralizer in \(\SL {13}\) is a \(2\)-group, again not our case. 

For larger values of \(\ell\) we note that, if \(\ell=2\), then the present inequality is the same as the one discussed for the case \(\dim_{\F_q} V=\ell\cdot (t^a+1)\), and the triples that have to be checked are \((2,5,2)\), \((2,7,2)\), \((2,9,2)\), \((2,11,2)\). Again by GAP \cite{GAP}, none of these triples corresponds to an existing module, the final contradiction for this case.  



\medskip
We omit the proof that \(V-\{0\}\neq V_{I_+}\cup V_{II}\), which is totally analogous to the previous one, and we focus next on the remaining claim that \(V-\{0\}\) is not covered by \(V_{I_-}\) and \(V_{I_+}\). Assuming the contrary, Table~\ref{Brauer3} yields that the dimension of \(V\) over \(\F_q\) is \(\ell\cdot(t^a-1)\) or \(\ell\cdot\left(\dfrac{t^a-1}{2}\right)\) (only for \(t\neq 2\)), and we will start by considering the former possibility. In view of the fact that both \(\dim_{\F_q}\cent V R\) and \(\dim_{\F_q}\cent V S\) are at most a half of \(\dim_{\F_q} V\), it is enough to analyze the inequality
\[q^{\ell\cdot(t^a-1)}-1>\dfrac{t^a\cdot(t^a+1)}{2}\cdot\left(q^{\ell\cdot\left(\frac{t^a-1}{2}\right)}-1\right)+\dfrac{t^a\cdot(t^a-1)}{2}\cdot\left(q^{\ell\cdot\left(\frac{t^a-1}{2}\right)}-1\right).\] As usual we reduce to \[q^{t^a-1}>t^a\cdot(t^a+1)\cdot q^{\frac{t^a-1}{2}},\] which is satisfied for \(t^a\geq 19\). For smaller values of \(t^a\), it can be checked (via the original equation for various values of \(\ell\) and taking into account Corollary~1.2 of \cite{L}) that the possible exceptions are the triples \((2,5,1)\), \((2,11,1)\), \((3,5,1)\), \((3,7,1)\), \((5,4,1)\) and \((2,5,2)\); among the corresponding modules, the first one has elements whose centralizers in \(\SL 5\) do not have a normal Sylow \(3\)-subgroup, the remaining four absolutely irreducible modules all have elements whose centralizers in the relevant group are \(2\)-groups, whereas the last one does not exists. The proof for this case is complete.

\medskip 
Finally, we focus on \(\dim_{\F_q} V=\ell\cdot\left(\dfrac{t^a-1}{2}\right)\), thus the inequality that has to be considered is \[q^{\ell\cdot\left(\frac{t^a-1}{2}\right)}-1>\dfrac{t^a\cdot(t^a+1)}{2}\cdot\left(q^{\ell\cdot\left(\frac{t^a-1}{4}\right)}-1\right)+\dfrac{t^a\cdot(t^a-1)}{2}\cdot\left(q^{\ell\cdot\left(\frac{t^a-1}{4}\right)}-1\right).\] With the usual argument, this reduces to \[q^{\frac{t^a-1}{2}}>t^{a}\cdot (t^a+1)\cdot q^{\frac{t^a-1}{4}}\] which is always satisfied for \(t^a\geq 47\). For smaller values of \(t^a\) (and for various \(\ell\)), from the original inequality we see that the triples to be checked are: \((2,t^a,1)\) for \(t^a\in\{7,9,17,23,25,31\}\), \((3,11,1)\), \((3,13,1)\), \((5,9,1)\), \((5,11,1)\), \((2,5,2)\), \((2,7,2)\), \((2,9,2)\), \((2,11,2)\), \((2,13,2)\), \((3,5,2)\), \((3,7,2)\), \((2,5,3)\), \((2,7,3)\), \((2,9,3)\) and \((2,5,4)\). None of these correspond to modules that satisfy our assumptions, the final contradiction that concludes the proof.
\end{proof} 

Next, the aforementioned result concerning actions of \(\SL{t^a}\) by automorphisms on a \(t\)-group.

\begin{theorem}
\label{TipoIeII}
Let \(T\) be a Sylow \(t\)-subgroup of \(G\cong\SL{t^a}\) (where \(t^a\geq 4\)) and, for a given odd prime divisor \(r\) of \(t^{2a}-1\), let \(R\) be a Sylow \(r\)-subgroup of \(G\). Assuming that \(V\) is a \(t\)-group such that \(G\) acts by automorphisms (not necessarily faithfully) on \(V\) and \(\cent V G=1\), consider the sets 
\[V_I=\{v\in V\mid {\textnormal{ there exists }} x\in G {\textnormal{ such that }}R^x\trianglelefteq\cent G v\},{\textnormal{ \emph{and}}}\] \[V_{II}=\{v\in V\mid {\textnormal{ there exists }} x \in G {\textnormal{ such that }}T^x\trianglelefteq\cent G v\}.\] Then the following conditions are equivalent.
\begin{enumeratei}
\item \(V_I\) and \(V_{II}\) are both non-empty and \(V-\{1\}=V_I\cup V_{II}\).
\item \(G\cong\SL {4}\), and \(V\) is an irreducible \(G\)-module of dimension \(4\) over \(\mathbb{F}_2\). More precisely, \(V\) is the restriction to $G$, embedded as \(\Omega_4^-(2)\) into ${\rm{SL}}_4(2)$, of the standard module of ${\rm{SL}}_4(2)$.
\end{enumeratei}
\end{theorem}

\begin{proof}
Set \(V^{\sharp}=V-\{1\}\) and, as in our hypothesis, assume \(V^{\sharp}=V_I\cup V_{II}\) where \(V_I\) and \(V_{II}\) are both non-empty; we start by observing that \(V_I\) and \(V_{II}\) are disjoint, because our assumptions ensure that no non-trivial element of \(V\) is centralized by the whole \(G\) and, on the other hand, no proper subgroup of \(G\) contains both a conjugate of \(R\) and a conjugate of \(T\) as a normal subgroup. By similar reasons, it is also clear that \(V_{II}\) is partitioned by  the subsets \(\cent{V^{\sharp}}{T^x}\) as \(x\) runs in \(G\).

As regards the set \(V_{I}\), we see that it is certainly covered by the union of the subsets \(\cent {V^{\sharp}}{R^x}\) as \(x\) runs in \(G\) (in fact, every element of \(V_{I}\) lies in precisely one of those subsets); however, only in the case when \(r\) is a divisor of \(t^a-1\), the subsets \(\cent {V^{\sharp}}{R^x}\) can have a non-empty intersection also with \(V_{II}\), and for our purposes it will be important to determine these intersections. 

To this end, assuming \(r\in\pi(t^a-1)\), we will show next that \(\cent {V^{\sharp}}{R}\cap V_{II}\) is the (disjoint) union of \(\cent {V^{\sharp}}{T_1R}\) and \(\cent {V^{\sharp}}{T_2R}\) where \(T_1\), \(T_2\) are suitable conjugates of \(T\). In fact, these \(T_1\) and \(T_2\) turn out to be the two Sylow \(t\)-subgroups of \(G\) that are normalized by \(R\) (see Lemma~\ref{PSL2}).

So, let \(T_0\in\syl t S\) be such that \(R\) lies in \(\norm G {T_0}\), and let \(v\) be an element of \(\cent {V^{\sharp}}{T_0R}\). Then \(v\) is in \(V_{II}\), as otherwise it would lie in \(V_I\) and \(\cent G v\) would contain a unique Sylow \(r\)-subgroup together with a Sylow \(t\)-subgroup of \(G\), which does not happen for any proper subgroup of \(G\). Also, it is clear that \(v\) is centralized by \(R\). 
As a consequence, if \(T_1\) and \(T_2\) are the two Sylow \(t\)-subgroups of \(G\) that are normalized by \(R\), then we get \(\cent{V^{\sharp}}{T_1R}\cup\cent{V^{\sharp}}{T_2R}\subseteq\cent{V^{\sharp}}{R}\cap V_{II}\). 
On the other hand, let \(v\) be an element of \(\cent {V^{\sharp}} R\cap V_{II}\), and let \(T_0\) be the (unique) Sylow \(t\)-subgroup of \(G\) that centralizes \(v\). As \(T_0\) is a normal subgroup of \(\cent G v\) and \(R\subseteq\cent G v\), we clearly get \(R\subseteq\norm G{T_0}\). This yields \(T_0\in\{T_1,T_2\}\) and,  together with the discussion in the previous paragraph, \(\cent{V^{\sharp}}{T_1R}\cup\cent{V^{\sharp}}{T_2R}=\cent{V^{\sharp}}{R}\cap V_{II}\), as wanted.

\smallskip
Still in the case \(r\in\pi(t^a-1)\), set \(|V|=t^d\), \(|\cent V R|=t^e\), \(|\cent V T|=t^f\) and, assuming \(R\subseteq\norm G T\) (as we may), \(|\cent V {TR}|=t^g\). Note that \(e\) and \(f\) are positive integers because \(V_I\) and \(V_{II}\) are non-empty; moreover, \(f\geq g\), and \(t^e>2t^g-1\). In view of the above discussion, we get the following equality. \[t^d-1=|V^{\sharp}|=\left(\sum_{T\in\syl t G}|\cent{V^{\sharp}} T|\right)+\left(\sum_{R\in\syl q G}|\cent {V^{\sharp}} R-V_{II}|\right)=\]
\[=(t^a+1)(t^f-1)+\dfrac{t^a(t^a+1)}{2}(t^e-2t^g+1),\] therefore,
\[2t^d+t^a+2t^{2a+g}+2t^{a+g}=2t^{a+f}+2t^f+t^{2a+e}+t^{2a}+t^{a+e}.\] It is not difficult to check that, by the uniqueness of the \(t\)-adic expansion, the above equality is never satisfied if \(t\) is odd. As regards the case \(t=2\), the equality is indeed satisfied if and only if \(d=2a\), \(g+1=e\) and \(a=f+1\), so, in particular, we get \(|V|=2^{2a}\). 

We focus next on the subgroup \(\Omega=\Omega(\zent V)\) of \(V\) generated by all the central elements of \(V\) having order \(2\). Since \(\Omega\) is a characteristic subgroup of \(V\) and it is an elementary abelian \(2\)-group, we can view it as an \(G\)-module over \(\mathbb{F}_2\) and we can consider an irreducible submodule \(W\) of it. Of course no non-trivial element of \(W\) is centralized by the whole \(G\), and we have \(W=(W\cap V_I)\cup(W\cap V_{II})\). Moreover, if \(W\cap V_{II}\) is empty, then \cite[Lemma~3.3]{ACDPS} yields a contradiction; on the other hand, if \(W\cap V_I\) is empty or if both \(W\cap V_I\) and \(W\cap V_{II}\) are non-empty, then we get \(|W|=2^{2a}(=|V|)\) via \cite[Lemma~5.2]{LWc} or via the discussion in this proof, respectively. In any case we conclude that \(V=W\) is an irreducible \(G\)-module of dimension \(2a\) over \(\mathbb{F}_2\). Now, by Lemma~3.12 in \cite{PR}, \(a\) is even (\(a=2b\), say) and \(V\) is necessarily the restriction to \(\mathbb{F}_2\) of the ``natural" (\(4\)-dimensional) \(\mathbb{F}_{2^b}\)-module for the group \(\Omega_4^-(2^b)\) or one of its Galois twists; furthermore, in this situation we get \(|\cent V T|=2^b\), i.e. \(f=a/2\). But we observed that \(a=f+1\), whence \(a=2\) and the proof that (a) implies (b) is complete in this case. The converse statement is also true, because the centralizers of the non-trival elements in the module described in (b) are isomorphic either to \(S_3\) or to \(A_4\).

\smallskip
It remains to show that condition (a) cannot hold if, under our hypotheses, \(r\) lies in \(\pi(t^a+1)\). In fact, using the notation introduced for the case \(r\in\pi(t^a-1)\), we get \[t^d-1=(t^a+1)(t^f-1)+\dfrac{t^a(t^a-1)}{2}(t^e-1),\] therefore
\[2t^d+t^a+t^{2a}+t^{a+e}=2t^{a+f}+2t^f+t^{2a+e}.\]  
Again by the uniqueness of the \(t\)-adic expansion, it is not difficult to see that the above equality is never satisfied.
\end{proof}

We conclude the section with Lemma~\ref{brauer}, that deals with one specific orbit condition in linear actions of \(\SL{t^a}\), for \(t\neq 2\), on modules over \(\mathbb{F}_2\).

\begin{lemma}\label{brauer}
Let \(V\) be a non-trivial irreducible module for \(G=\SL{t^a}\) over the field \(\F_2\), where \(t\) is an odd prime with \(t^a>3\). Let \(r\neq t\) be a prime in \(\{3,5\}\). Then there exists \(v\in V\) such that \(\cent G v\) does not contain any element of order \(r\), except in the following cases: \(t^a=5\) and \(\dim_{\F_2} V=4\), or \(t^a=7\) and \(\dim_{\F_2} V=3\) (in both cases \(r=3\)).
\end{lemma}

\begin{proof}
We will prove the statement as follows. Given an element \(x\in G\) of order \(r\), we will establish an upper bound for the dimension over \(\F_2\) of the subspace \(\cent V x\); we will then see that, with the possible exceptions mentioned in the statement, this dimension is too small to get a covering of \(V\) with subspaces of the form \(\cent V x\) where \(x\) ranges over the elements of order \(r\) in \(G\). 

The bound for \(\dim_{\F_2}\cent V x\) can be obtained by the same argument used in Lemma~\ref{brauer2}. Observing that \(V\) is in fact a module for \(S=\PSL{t^a}\) (because the element of order \(2\) of \(G\) acts trivially on \(V\)), we consider the irreducible \(2\)-Brauer characters of \(S\).  
In view of the discussion of \cite[Section VIII]{B}, it turns out that every non-principal irreducible \(2\)-Brauer character \(\theta\) of \(S\) is unique in its \(2\)-block and it has a lift \(\chi\in\irr S \), except when \(t^a-1\) is divisible by \(2^2\) and \(\theta\) lies in the principal \(2\)-block of \(S\); in the latter case, however, there exists \(\gamma\in\irr S\) such that \(\gamma-1_S\) is a lift for \(\theta\). In particular, the degrees of the irreducible \(2\)-Brauer characters of \(S\) are \(t^a-1\), \((t^a-1)/2\) and \(t^a+1\).
Now we can refer to the ordinary character table of \(S\) (see for example~\cite[Pages~80, 82]{B}) and compute \([1_{\langle x\rangle},\chi_{\langle x\rangle}]\), where \(\chi\in\irr S\) ranges over a set of lifts for the non-principal irreducible \(2\)-Brauer characters of \(S\). In Table~\ref{lifts3} and Table~\ref{lifts5} the parameter \(\ell\) denotes (as in Lemma~\ref{brauer2}) the composition length of the \(\F[G]\)-module \(V\otimes \F\), where \(\F\) is a splitting field for \(G\) and all its subgroups; the maximal value of \([1_{\langle x\rangle},\chi_{\langle x\rangle}]\) is shown in the second and third row of the tables (dealing with the cases \(r\mid t^a-1\) and \(r\mid t^a+1\), respectively). 

\begin{table}[htbp]
\caption{\label{lifts3} Maximal dimension of the centralizer of an element of order \(3\).}
\renewcommand{\arraystretch}{.95}
$$
\begin{array}{llll} \hline\hline\\ 
\dim_{\F_2} V\qquad\qquad & {\ell\cdot (t^a-1)}\qquad\quad&\ell\cdot\left(\dfrac{t^a-1}{2}\right)\qquad\quad& \ell\cdot(t^a+1)\\ \\   
\hline
\\ 
\dim_{\F_2}\cent V x&\ell\cdot\left( \dfrac{t^a-1}{3} \right)&\ell\cdot\left( \dfrac{t^a-1}{6} \right)&\ell\cdot\left( \dfrac{t^a+5}{3} \right)\\
\text{\rm{for }}t^a\equiv 1 {\textnormal{ (mod } 3)}&&&\\ \\
\hline
\\
\dim_{\F_2}\cent V x&\ell\cdot\left( \dfrac{t^a+1}{3} \right)&\ell\cdot\left( \dfrac{t^a-5}{6} \right)&\ell\cdot\left( \dfrac{t^a+1}{3} \right)\\
\text{\rm{for }}t^a\equiv -1 {\textnormal{ (mod } 3)}&&&\\ \\
\hline\hline
\\ \\
\end{array}
$$
\end{table}

\begin{table}[htbp]
\caption{\label{lifts5} Maximal dimension of the centralizer of an element of order \(5\).}
\renewcommand{\arraystretch}{.95}
$$
\begin{array}{llll} \hline\hline\\ 
\dim_{\F_2} V\qquad\qquad& {\ell\cdot (t^a-1)}\qquad\quad&\ell\cdot\left(\dfrac{t^a-1}{2}\right)\qquad\quad& \ell\cdot(t^a+1)\\ \\   
\hline
\\ 
\dim_{\F_2}\cent V x&\ell\cdot\left( \dfrac{t^a-1}{5} \right)&\ell\cdot\left( \dfrac{t^a-1}{10} \right)&\ell\cdot\left( \dfrac{t^a+9}{5} \right)\\
\text{\rm{for }}t^a\equiv 1 {\textnormal{ (mod } 5)}&&&\\ \\
\hline
\\
\dim_{\F_2}\cent V x&\ell\cdot\left( \dfrac{t^a+1}{5} \right)&\ell\cdot\left( \dfrac{t^a-9}{10} \right)&\ell\cdot\left( \dfrac{t^a+1}{5} \right)\\
\text{\rm{for }}t^a\equiv -1 {\textnormal{ (mod } 5)}&&&\\ \\
\hline\hline
\end{array}
$$
\end{table}

Next, assume that every element of \(V\) is centralized by an element of order \(r\) of \(G\); then, denoting by \(k\) the number of Sylow \(r\)-subgroups of \(G\) and choosing a non-trivial element \(x_i\) (\(i\in\{1,\ldots, k\}\)) from each of those subgroups, we have  \[V-\{0\}=\bigcup_{i=1}^k(\cent V {x_i}-\{0\}).\] Observe that \(k=t^a\cdot(t^a\pm 1)\), where the plus or minus sign occurs according to whether \(t^a\equiv 1{\textnormal{ (mod }} r)\) or \(t^a\equiv -1{\textnormal{ (mod }} r)\) respectively. So, setting \(m\) to be the maximal dimension of \(\cent V x\) corresponding to each \(d=\dim_{\F_2}V\) as shown in Table~\ref{lifts3} and Table~\ref{lifts5}, we consider the inequality 
\begin{equation}\label{ineq}
2^d-1> t^a\cdot(t^a+ 1)\cdot(2^m-1),
\end{equation}
and we discard all the modules whose corresponding pair \((t^a,d)\) satisfies  Inequality~(1) (the appropriate value of \(m\) is deduced by the tables from the values of \(t^a\) and \(d\)).  

Note that (1) is true whenever we have 
\begin{equation} \label{ineq2}
2^d> t^a\cdot(t^a+ 1)\cdot 2^m;
\end{equation} 
furthermore, writing \(d=\ell d_0\), \(m=\ell m_0\), and applying the function \(\log_2\) to both sides, it is also easy to see that Inequality~(2) is in turn satisfied if \(2^{d_0}> t^a\cdot(t^a+ 1)\cdot 2^{m_0}\) is. In other words, if an absolutely irreducible module (i.e. an irreducible module for which the parameter \(\ell\) is \(1\)) has a corresponding pair \((t^a,d_0)\) satisfying~(2), then we can discard every irreducible module corresponding to a pair of the kind \((t^a,\ell d_0)\).

Now, the list of pairs \((t^a,d)\) that do not satisfy (2) and that correspond to absolutely irreducible modules is the following.
\begin{enumeratei}
\item For \(r=3\): \((5,2)\), \((5,4)\), \((5,6)\), \((7, 3)\), \((7,6)\), \((7, 8)\), \((11, 5)\), \((11, 10)\), \((13, 6)\), \((17, 8)\), \((19, 9)\), \((23, 11)\), \((25, 12)\).
\item For \(r=5\): \((9,4)\), \((9,8)\), \((11,5)\), \((19,9)\).
\end{enumeratei}
This can be refined further using Corollary~1.2 of \cite{L} (so, discarding the modules on which the action of \(G\) generates regular orbits); as a result, the pairs that still need to be analyzed for the case \(\ell=1\) are the following.
\begin{enumeratei}
\item For \(r=3\): \((5,4)\), \((7,3)\), \((7, 8)\), \((11, 10)\), \((17, 8)\), \((23, 11)\), \((25, 12)\).
\item For \(r=5\): \((9,4)\).
\end{enumeratei}
A direct computation with GAP \cite{GAP} shows that, in all of the above cases, there are elements of the relevant module whose centralizer in \(G\) does not have an order divisible by \(r\), except for the pairs \((5,4)\) and \((7,3)\) as claimed in the statement.

As regards the cases when \(\ell>1\), only three pairs do not satisfy inequality (2), the value of \(\ell\) being \(2\) for all of them: these are \((5,4)\), \((5,8)\) and \((7,6)\). The non-absolutely irreducible pair \((5,4)\) is associated to the natural module of \(\SL 4\cong \PSL 5\) and obviously does not satisfy the condition about the centralizers, whereas the pairs \((5,8)\) and \((7,6)\) are not associated to any existing module. 

On the other hand, if \(V\) is an absolutely irreducible module corresponding to the pairs \((5,4)\) or \((7,3)\), then every element of \(V\) is actually centralized by an element of order \(3\) (the sets of orbit sizes are \(\{5,10\}\) and \(\{7\}\) respectively). The proof is complete.
\end{proof}

\section{The structure of the solvable residual}

Let \(G\) be a non-solvable group having a normal section isomorphic to \(\PSL{t^a}\), where \(t\) is an odd prime with \(t^a> 5\), and assume that \(\Delta(G)\) is connected with a cut-vertex. Our aim in this section is to describe the structure of the solvable residual \(K\) of \(G\). In particular we will see that, except for one sporadic case, either we have \(K\cong\PSL{t^a}\), or  \(K\cong\SL{t^a}\), or \(K\) contains an abelian minimal normal subgroup \(L\) of \(G\) such that \(K/L\cong\SL{t^a}\) and \(L\) is the natural module for \(K/L\). We will actually prove that the dual group of \(L\) is the natural module for \(K/L\), but the desired conclusion follows taking into account that this module is self dual. 
  
\begin{theorem}\label{MainAboutK}
Assume that the group \(G\) has a composition factor isomorphic to $\PSL{t^a}$, for an odd prime \(t\) with \(t^a> 5\), and let \(p\) be a prime number. Assume also that \(\Delta(G)\) is connected with cut-vertex \(p\). Then, denoting by \(K\) the solvable residual of \(G\), one of the following conclusions holds.
\begin{enumeratei} 
\item \(K\) is isomorphic to \(\PSL{t^a}\) or to \(\SL{t^a}\); 
\item \(K\) contains a minimal normal subgroup \(L\) of \(G\) such that \(K/L\) is isomorphic to \(\SL{t^a}\) and \(L\) is the natural module for \(K/L\).
\item \(t^a=13\) and \(p=2\). \(K\) contains a minimal normal subgroup \(L\) of \(G\) such that \(K/L\) is isomorphic to \(\SL{13}\), and \(L\) is isomorphic to one of the two \(6\)-dimensional irreducible modules for \(\SL{13}\) over \(\F_3\). 
\end{enumeratei}
\end{theorem}

Our analysis concerning Theorem~\ref{MainAboutK} splits in two parts, depending on whether one of the sets of primes \(\pi(t^a-1)\), \(\pi(t^a+1)\) reduces to the prime \(2\) or not. Theorem~\ref{pallasgonfia}, which is introduced by Proposition~\ref{ThreeVertices}, deals with the former situation.

\begin{proposition} \label{ThreeVertices} Set \(G\simeq\PSL{t^a}\), where \(t\) is an odd prime, and assume that one among the sets of primes \(\pi_{-}=\pi(t^a-1)-\{2\}\) and \(\pi_{+}=\pi(t^a+1)-\{2\}\) is empty. Then the following conclusions hold.
\begin{enumeratei}
\item \(a=1\), unless \(t^a=9\).
\item If \(|\pi(G)|=3\), then \(t^a\in\{5,7,9,17\}\).
\end{enumeratei}
\end{proposition} 

\begin{proof}
Our assumptions, together with Proposition~3.1 of \cite{MW}, yield that \(a=1\) and \(t\) can be written as \(2^k-1\) or \(2^k+1\) for a suitable integer \(k>2\), with the only exception of the case \(t^a=9\) (if \(t=2^k-1\), then \(t\) is a Mersenne prime and \(k\) is a prime number, whereas if \(t=2^k+1\), then \(t\) is a Fermat prime and \(k\) is a \(2\)-power). This proves conclusion~(a).

Assuming now in addition that \(|\pi(G)|=3\), we write \(\pi(G)=\{2,t,u\}\) and we consider first the case when \(\pi_{+}\) is empty, so \(t=2^k-1\) for a suitable prime number \(k\). Then \(t-1=2^k-2=2\cdot (2^{k-1}-1)\) is divisible only by \(2\) and \(u\), which means that \(u=2^{k-1}-1\) is in turn a Mersenne prime and \(k-1\) is a prime number. This can happen only for \(k=3\), thus \(t=7\). 
On the other hand, if \(\pi_{-}\) is the empty one, then either \(t^a=9\) or \(t=2^k+1\) is a Fermat prime and \(k\) is a \(2\)-power. Consider the latter case: the fact that \(\pi_{+}\) consists of the single prime \(u\) yields that \((2^{k-1}+1)=(t+1)/2\) is a power of \(u\), so either \(u=2^{k-1}+1\) is in turn a Fermat prime (and \(k-1\) is a \(2\)-power as well) or \(u^2=9=2^{k-1}+1\). If both \(k\) and \(k-1\) are \(2\)-powers, then necessarily we get \(k=2\) and \(t=5\); it remains then the possibility \(u^2=9\), thus \(k-1=3\) and \(t=2^k+1=2^4+1=17\). 
\end{proof}

\begin{theorem}\label{pallasgonfia}
Assume that the group \(G\) has a composition factor isomorphic to $\PSL{t^a}$, for an odd prime \(t\) with \(t^a> 5\), and let \(p\) be a prime number. Assume also that \(\Delta(G)\) is connected with cut-vertex \(p\), and that one among the sets of primes \(\pi_{-}=\pi(t^a-1)-\{2\}\) and \(\pi_{+}=\pi(t^a+1)-\{2\}\) is empty.
Then, denoting by \(K\) the solvable residual of \(G\), one of the following conclusions holds.
\begin{enumeratei} 
\item \(K\) is isomorphic to \(\PSL{t^a}\) or to \(\SL{t^a}\); 
\item \(K\) contains a minimal normal subgroup \(L\) of \(G\) such that \(K/L\) is isomorphic to \(\SL{t^a}\) and \(L\) is the natural module for \(K/L\).
\end{enumeratei}
\end{theorem}

\begin{proof}
Let $R$ be the solvable radical of $G$. Observe that, by Theorem~\ref{0.2}, the factor group \(G/R\) is an almost-simple group (with socle isomorphic to \(\PSL{t^a}\)) and \(\V G=\V{G/R}\cup\{p\}\); furthermore,  Proposition~\ref{ThreeVertices} yields that \(a=1\) with the only exception of the case \(t^a=9\), therefore \(\V G\) consists of \(p\) and the primes dividing the order of the socle of \(G/R\). As the subgraph of \(\Delta(G)\) induced by \(\V G-\{t,p\}\) is then a complete subgraph, it is clear that \(p\neq t\) and that the set of neighbours of \(t\) in \(\Delta(G)\) is \(\{p\}\). As another consequence of the fact that \(a=1\) or \(t^a=9\), the socle of \(G/R\)  does not have any proper subgroup of type~(iv) except when \(t^a=9\), in which case the relevant subgroups (isomorphic to \(A_4\) or \(S_4\)) are in fact also of type~(iii).

Consider now the solvable residual \(K\) of \(G\) and assume \(t^a\neq 9\). By Lemma~\ref{InfiniteCommutator} we can assume that, for a suitable non-trivial normal subgroup \(L\) of \(K\), we have \(K/L\cong\PSL{t^a}\) or \(K/L\cong\SL{t^a}\); moreover, for every non-principal irreducible character \(\lambda\) of \(L/L'\), the inertia subgroup \(I_K(\lambda)\) is a proper subgroup of \(K\) (and the same clearly holds also for \(I_K(\lambda)N\), where \(N/L=\zent{K/L}\)). 

We claim that the conclusions of Lemma~\ref{InfiniteCommutator} actually hold even if \(t^a=9\). In fact, as in the proof of that lemma, consider a chief factor \(K/N\cong\PSL{9}\) of \(G\) and suppose (as we may) \(N\neq 1\); if \(\mu\) is a non-principal \(K\)-invariant irreducible character of \(N/N'\) and \(L=\ker\mu\), then (\(L\trianglelefteq K\) and) \(N/L\) lies in the Schur multiplier of \(K/N\), that has order \(6\). But if \(K/L\) is a central extension of \(N/L\) by \(K/N\cong\PSL{9}\) such that \(3\) divides \(|N/L|\), then it can be checked that the set of irreducible character degrees of \(K/L\) contains both \(6\) and \(15\), thus \(\Delta(K/L)\) is a triangle of vertices \(2,3,5\) and \(\Delta(G)\) does not satisfy the hypothesis. The conclusion so far is that \(K/L\cong\SL{9}\), and again we can assume \(L\neq 1\). Now, since the Schur multiplier of \(\SL{9}\) has order \(3\), but the Schur cover of \(\SL{9}\) has irreducible characters of degree \(6\) and \(15\), arguing as above we see that \(L/L'\) does not have any \(K\)-invariant irreducible character.

Given that, the proof is organized as follows. First, we will show that in the present setting \(K/L\) is isomorphic to \(\SL{t^a}\) and \(L/L'\) is the natural module for \(K/L\); we will treat separately four cases, depending on whether \(|\pi(K/N)|\) is larger than \(3\) or not, and whether \(p\neq 2\) or \(p=2\). After that, we will prove that \(L'\) is in fact trivial. This, together with the observation that in this situation \(L=\oh t K\) (therefore \(L\) is actually a normal subgroup of \(G\)) and \(L\) is already a minimal normal subgroup of \(K\), will conclude the proof.

\medskip
Consider first the case when \(|\pi(K/N)|\) is larger than \(3\), i.e. the non-empty set among \(\pi_{-}\) and \(\pi_{+}\) does not consist of a single prime (observe that the case \(t^a=9\) is then excluded, so we always have \(a=1\)), and \(p\neq 2\). In what follows, we denote by \(\pi\) the non-empty set among \(\pi_+\) and \(\pi_-\) (in other words, we set \(\pi=\pi_+\cup\pi_-\)). 

Let \(\lambda\) be a non-principal irreducible character of \(L/L'\), and assume that \(t\) divides \(|K:I_K(\lambda)|\). Then \(I_K(\lambda)/L\) must contain a Sylow \(2\)-subgroup of \(K/L\), as otherwise \(t\) would be adjacent to \(2\) in \(\Delta(G)\) (recall that \(p\neq 2\) is the only neighbour of \(t\) in \(\Delta(G)\)); in particular, \(N\) is contained in \(I_K(\lambda)\). In this situation, if \(I_K(\lambda)/N\) is of type~(i), then it is easy to see that \(t\) is adjacent in \(\Delta(G)\) to every prime in \(\pi\), which implies \(\pi=\{p\}\) and \(|\pi(K/N)|=3\), not the case we are considering. It is also clear that \(I_K(\lambda)/N\) cannot be of type~(ii). Finally, if \(I_K(\lambda)/N\) is of type (iii), then the largest power of \(2\) that divides \(|K/N|\) is \(2^2\) or \(2^3\); since it is easily seen that, writing \(t=2^k\pm 1\), the \(2\)-part of \(|K/N|\) is \(2^k\), we get either  \(k=2\) (but recall that \(k>2\) in the present situation) or \(k=3\), \(t=7\) and \(\V G=\{2,3,7\}\), in any case a contradiction.

On the other hand, assume that \(t\) does not divide \(|K:I_K(\lambda)|\). Then \(I_K(\lambda)N/N\) is not of type~(i); moreover, if this factor group is of type~(iii), then we get \(t=3\) or \(t=5\), against our hypothesis. The only possibility is then the type (ii): in particular, \emph{\(I_K(\lambda)/L\) contains a Sylow \(t\)-subgroup of \(K/L\) as a normal subgroup.}
Since this holds for every non-principal irreducible character \(\lambda\) of \(L/L'\), by Lemma~\ref{SL2Nq} we conclude that \(K/L\cong\SL{t}\) and \(L/L'\) is the natural module, as desired.

\smallskip
Still assuming \(|\pi(K/N)|>3\), we move now to the case \(p=2\).

Again, let \(\lambda\) be a non-principal irreducible character of \(L/L'\). If \(t\) divides \(|K:I_K(\lambda)|\), then the factor group \(I_K(\lambda)N/N\) is clearly not a subgroup of type (ii) with non-trivial \(t\)-part of \(K/N\). But this factor group is not of type (iii) as well: in fact, if \(I_K(\lambda)N/N\) is isomorphic to \(A_4\) or \(S_4\), then \(\pi\) would consist of the element \(3\) only, whereas if \(I_K(\lambda)N/N\cong A_5\), then  (by Gallagher's theorem or by the theory of character triples, according to whether \(\lambda\) has an extension to its inertia subgroup or not) we see that \(t\) is adjacent to \(3\) in \(\Delta(G)\), anyway a contradiction. Therefore \(I_K(\lambda)N/N\) must be of type (i), containing a  Hall \(\pi\)-subgroup of \(K/N\) as a normal subgroup; as easily checked, we then have that \(I_K(\lambda)/L\) contains a Hall \(\pi\)-subgroup of \(K/L\) as a normal subgroup.

On the other hand, if \(t\) does not divide \(|K:I_K(\lambda)|\), then \(I_K(\lambda)N/N\) is a subgroup of type (ii) of \(K/N\), and \(I_K(\lambda)/L\) contains a Sylow \(t\)-subgroup of \(K/L\) as a normal subgroup.

Now, assume that \(L/L'\) is not a \(t\)-group. Then there exists a chief factor \(L/Y\) of \(K\) that is a \(q\)-group for a suitable prime \(q\neq t\). Denoting by \(V\) the dual group of \(L/Y\), the discussion carried out so far yields that for every non-trivial element \(\lambda\) of \(V\), the factor group \(I_K(\lambda)/L\) contains either a unique Hall \(\pi\)-subgroup or a unique Sylow \(t\)-subgroup of \(K/L\). This is against Theorem~\ref{TipoIeIIPieni}, thus we conclude that \(L/L'\) is a \(t\)-group,
and we are in a position to apply Theorem~\ref{TipoIeII} obtaining that the subgroups \(I_K(\lambda)/L\) (for \(\lambda\) a non-principal character in \(\irr{L/U}\)) are all of the same kind: either they all contain a unique Hall \(\pi\)-subgroup of \(K/L\) or they all contain a unique Sylow \(t\)-subgroup of \(K/L\). However, by Lemma~\ref{SL2Nq} (and recalling \cite[Lemma~4]{Z}) the former condition is impossible and \(L/L'\) is the natural module for \(K/L\cong\SL{t}\), as desired. 

\smallskip
Now, let us consider the case when \(|\pi(K/N)|=3\), so \(\pi(K/N)=\{2,u,t\}\) for a suitable prime \(u\). Recall that, by Proposition~\ref{ThreeVertices}, \(K/N\) is then isomorphic to one of the following groups: \(\PSL 7\), \(\PSL 9\) or \(\PSL{17}\).

We assume first that, under the hypothesis \(|\pi(K/N)|=3\), we have \(p\neq 2\). Thus, if \(\lambda\) is a non-principal irreducible character of \(L/L'\) and \(t\) divides \(|K:I_K(\lambda)|\), then \(2\) cannot divide the degree of any irreducible character of \(K\) lying over \(\lambda\). In particular, \(K/L\) should have an abelian Sylow \(2\)-subgroup by Theorem~A of \cite{NT}, and this is not the case. Looking at \cite{ATLAS}, the only remaining possibility is that \(I_K(\lambda)/L\) contains a Sylow \(t\)-subgroup of \(K/L\) as a normal subgroup. This holds for every non-principal \(\lambda\in\irr{L/L'}\), and again by Lemma~\ref{SL2Nq} we have that \(L/L'\) is the natural module for \(K/L\cong\SL {t^a}\).

\smallskip
Finally let \(|\pi(K/N)|=3\) and \(p=2\). Assume first \(K/N\cong\PSL 7\): in this case we have \(\V G=\{2,3,7\}\) and the only non-adjacency in \(\Delta(G)\) is between \(3\) and \(7\). Now if, for every non-principal \(\lambda\) in \(\irr{L/L'}\), we have that \(7\) does not divide \(|K:I_K(\lambda)N|\), then \(I_K(\lambda)/L\) contains a Sylow \(7\)-subgroup of \(K/L\) as a normal subgroup and \(L/L'\) is the natural module for \(K/L\cong\SL 7\) (by Lemma~\ref{SL2Nq}); therefore, in view of the structure of the maximal subgroups of \(\PSL{7}\), we can assume that there exists \(\lambda_0\in\irr{L/L'}\) such that \(I_K(\lambda_0)N/N\) is isomorphic to a subgroup of \(S_4\). Of course the order of \(I_K(\lambda_0)N/N\) must be a multiple of \(3\), so either \(I_K(\lambda_0)/L\) contains a Sylow \(3\)-subgroup of \(K/L\) as a normal subgroup, or it has a normal section isomorphic to \(A_4\). In the latter case \(\lambda_0\) cannot extend to \(I_K(\lambda_0)\), as otherwise we would get the adjacency between \(3\) and \(7\) by Clifford's theory; hence \(|L/L'|\) is an even number, and there exists a chief factor \(L/Y\) of \(K\) that is a \(2\)-group. But it can be checked (via GAP \cite{GAP}, for instance) that all the three irreducible modules of \(\SL 7\) over \(\F_2\) (that have dimensions \(3\), \(3\), \(8\) respectively) produce factor groups \(K/Y\) having irreducible characters of degree divisible by \(21\), and this is a contradiction. By Ito-Michler's theorem and Gallagher's theorem we conclude that, for every non-principal \(\lambda\in\irr{L/L'}\), \(I_K(\lambda)/L\) contains either a Sylow \(7\)-subgroup or a Sylow \(3\)-subgroup of \(K/L\) as a normal subgroup. If \(L/L'\) is not a \(7\)-group, then we can consider a chief factor \(L/Y\) of \(K\) such that \(L/Y\) is a \(q\)-group for a suitable prime \(q\neq 7\), and Theorem~\ref{TipoIeIIPieni} applied to the action of \(K/L\) on the dual group of \(L/Y\) yields a contradiction. Therefore \(L/L'\) is a \(t\)-group, and now by Theorem~\ref{TipoIeII} (together with Lemma~\ref{SL2Nq}) we get that \(L/L'\) is in fact the natural module for \(K/L\cong\SL 7\).

Assume now \(K/N\cong\PSL 9\), thus \(\V G=\{2,3,5\}\) and the only non-adjacency in \(\Delta(G)\) is between \(3\) and \(5\). If, for every non-principal \(\lambda\) in \(\irr{L/L'}\), we have that \(3\) does not divide \(|K:I_K(\lambda)N|\), then \(I_K(\lambda)/L\) contains a Sylow \(3\)-subgroup of \(K/L\) as a normal subgroup and \(L/L'\) is the natural module for \(K/L\cong\SL 9\); therefore, we can assume that there exists \(\lambda_0\in\irr{L/L'}\) such that \(I_K(\lambda_0)N/N\) is isomorphic to a subgroup of \(A_5\). Of course the order of \(I_K(\lambda_0)N/N\) must be a multiple of \(5\), so either \(I_K(\lambda_0)/L\) contains a Sylow \(5\)-subgroup of \(K/L\) as a normal subgroup, or it is isomorphic to \(A_5\). In the latter case \(\lambda_0\) cannot extend to \(I_K(\lambda_0)\), as otherwise we would get the adjacency between \(3\) and \(5\) by Clifford's theory; hence \(|L/L'|\) is an even number, and there exists a chief factor \(L/Y\) of \(K\) that is a \(2\)-group. But it can be checked (via GAP \cite{GAP}, for instance) that all the three irreducible modules of \(\SL 9\) over \(\mathbb{F}_2\) (that have dimensions \(4\), \(4\), \(16\) respectively) produce factor groups \(K/Y\) having irreducible characters of degree divisible by \(15\), and this is a contradiction. We conclude that, for every non-principal \(\lambda\in\irr{L/L'}\), \(I_K(\lambda)/L\) contains either a Sylow \(3\)-subgroup or a Sylow \(5\)-subgroup of \(K/L\) as a normal subgroup; the same argument as in the paragraph above shows that \(L/L'\) is the natural module for \(K/L\cong\SL 9\).

Finally, consider the case \(K/N\cong\PSL {17}\), hence \(\V G=\{2,3,17\}\) and the only non-adjacency in \(\Delta(G)\) is between \(3\) and \(17\). If, for every non-principal \(\lambda\) in \(\irr{L/L'}\), we have that \(17\) does not divide \(|K:I_K(\lambda)N|\), then \(I_K(\lambda)/L\) contains a Sylow \(17\)-subgroup of \(K/L\) as a normal subgroup and \(L/L'\) is the natural module for \(K/L\cong\SL {17}\); therefore, we can assume that there exists \(\lambda_0\in\irr{L/L'}\) such that \(I_K(\lambda_0)N/N\) has order not divisible by \(17\), but divisible by \(3^2\). An inspection of the subgroups of \(\PSL {17}\) yields that \(I_K(\lambda_0)N/N\) is either a cyclic group of order \(3^2\) or it is a dihedral group of order \(18\), so, in any case it contains a Sylow \(3\)-subgroup of \(K/N\) as a normal subgroup and we can get the conclusion that \(L/L'\) is the natural module for \(K/L\cong\SL{17}\) as in the previous cases.

\medskip
The proof that \(L/L'\) is the natural module for \(K/N\cong\SL{t^a}\) is then concluded.
It remains to show that, setting \(U=L'\), we have \(U=1\). This will also imply that \(L\) is normal in \(G\), because if \(U=1\) then we have \(L=\oh t K\). Again we will treat separately the cases \(p\neq 2\) and \(p=2\). 

\smallskip
Assume \(p\neq 2\) and, aiming at a contradiction, assume \(U\neq 1\). Since \(U/U'\) is an abelian (non-trivial) normal subgroup of \(L/U'\) having \(t\)-power index, any non-linear irreducible character \(\phi\) of \(L/U'\) is such that \(t\) divides \(\phi(1)\). Now, we know that every irreducible character of \(K\) lying over \(\phi\) must have odd degree, and so \cite[Theorem~A]{NT} yields that \(K/L\) has abelian Sylow \(2\)-subgroups. This is clearly a contradiction, and the proof is complete.

\smallskip
Finally, let us consider the case \(p=2\) and, again assuming \(U\neq 1\), let us take a subgroup \(Z\) of \(K\) such that \(U/Z\) is a chief factor of \(K\). We treat three situations, that are exhaustive, and that all lead to a contradiction.

{\bf{(i)}} \(U/Z\not\leq\zent{L/Z}.\) 

\noindent(Note that, in this case, \(U/Z\) is a faithful \(K/U\)-module.) Consider the normal subgroup \(\cent{U/Z}{L/U}\) of \(K/Z\); since \(U/Z\) is a chief factor of \(K\) and it is not centralized by \(L/U\), we deduce that \(\cent{U/Z}{L/U}\) is trivial. We are in a position to apply the proposition appearing in the Introduction of \cite{Cu}, which ensures that the second cohomology group \({\rm {H}}^2(K/U,U/Z)\) is trivial, and therefore \(K/Z\) is a split extension of \(U/Z\); in particular, every irreducible character of \(U/Z\) extends to its inertia subgroup in \(K/Z\). Now, let \(\lambda\) be any non-principal character in \(\irr{U/Z}\): if \(\xi\in\irr{L/Z}\) lies over \(\lambda\), then \(\xi(1)\neq 1\) (as \(U/Z=(L/Z)'\)) and \(\xi(1)\) is a divisor of \(|L/U|=t^2\). In particular, every irreducible character of \(K\) lying over \(\lambda\) has a degree divisible by \(t\). Since \(\lambda\) extends to \(I_K(\lambda)/Z\), our assumptions together with Gallagher's theorem imply that \(I_K(\lambda)/U\) contains a unique Hall \(\pi\)-subgroup of \(K/U\). But this yields a contradiction via, for example, Proposition~3.13 of \cite{DPSS}; in fact, according to that result, \(K/U\) should have a cyclic solvable radical (whereas \(L/U=\oh t{K/U}\) is non-cyclic).

{\bf{(ii)}} \(U/Z\leq\zent{L/Z}\), but \(U/Z\not\leq\zent{K/Z}.\)

\noindent  
Let \(\lambda\) be a non-principal character in \(\irr{U/Z}\), and let \(\xi\in\irr{L/Z}\) be a character lying over \(\lambda\). Clearly \(\xi(1)\) is a multiple of \(t\) and, assuming for the moment \(t^a\neq 9\), \(\xi\) extends to its inertia subgroup in \(K\) because \(K/L\) has cyclic Sylow \(t\)-subgroups (\cite[8.16, 11.22, 11.31]{Is}). It follows by Gallagher's theorem that \(I_K(\xi)\) contains a Hall \(\pi\)-subgroup of \(K/L\) as a normal subgroup. 
Now, \(\xi_U=\xi(1)\lambda\) and \(\xi(1)^2\leq t^2\), therefore \(\xi(1)=t\) and \(\lambda\) is fully ramified with respect to \(L/U\). In particular, \(\xi\) vanishes outside \(U\), thus \(I_K(\lambda)=I_K(\xi)\) and so \(I_K(\lambda)/L\) contains a Hall \(\pi\)-subgroup of \(K/L\) as a normal subgroup. This is against Lemma~\ref{SL2Nq}, and we are done under the additional assumption \(t^a\neq 9\).

As regards the remaining case, we proceed as follows. Observe first that \(U/Z\) is a \(3\)-group, as otherwise \(L/Z\) would be the direct product of \(U/Z\) with an abelian \(3\)-group, and it would be abelian. Moreover, if  \(\lambda\) is a non-principal character in \(\irr{U/Z}\), then any irreducible character of \(K\) lying over \(\lambda\) has a degree divisible by \(3\); as a consequence, \(|K:I_K(\lambda)|\) is not divisible by \(5\). But a direct computation on the non-trivial irreducible modules for \(\SL{9}\) over \(\F_{3}\) (there are five of them, of dimensions \(4\), \(4\), \(6\), \(9\), \(12\)) shows that in every such module there are elements lying in orbits of size divisible by \(5\), the final contradiction that concludes the proof for this case. 

{\bf{(iii)}} \(U/Z\leq\zent{K/Z}.\)

\noindent Let \(\lambda\) be a non-principal character in \(\irr{U/Z}\). We have that \((L/Z,U/Z,\lambda)\) is a character triple for which the factor group \((L/Z)/(U/Z)\cong L/U\) is abelian, thus we can apply Lemma~2.2 of \cite{Wo}: there exists a unique subgroup \(W/Z\) of \(L/Z\), containing \(U/Z\), maximal with respect to the fact that \(\lambda\) has an \(L/Z\)-invariant extension to \(W/Z\). By the uniqueness of this \(W/Z\) and by the fact that \(\lambda\) is invariant in \(K\), it follows that \(W/Z\) is normal in \(K/Z\) and, since it is properly contained in \(L/Z\) (because \(\lambda\) does not extend to \(L\)), we get \(W=U\). Now part (b) of the same lemma yields that \(\lambda\) is fully ramified with respect to \(L/U\), and therefore the unique \(\xi\) in \(\irr{L/Z\mid\lambda}\) is such that \(I_K(\xi)=I_K(\lambda)=K\). 

Now, if \(\xi\) (whose degree is divisible by \(t\)) extends to \(K\), then we reach a contradiction via Gallagher's theorem. This always happens when the Schur multiplier of \(K/L\) is trivial, i.e. when \(t^a\neq 9\). As for the remaining case, working with character triples, we see that there exists an irreducible character of \(K\) lying over \(\xi\) whose degree is divisible by \(5\) (actually, by \(15\)) even if \(\xi\) does not extend to \(K\). This is impossible in our setting, and the proof is complete.
\end{proof}

We consider now, still for \(t\) odd and \(t^a\neq 5\), the complementary situation with respect to previous result. This will conclude the proof of Theorem~\ref{MainAboutK}.

\begin{theorem}\label{pallepiene}
Assume that the group \(G\) has a composition factor isomorphic to $\PSL{t^a}$, for an odd prime \(t\) with \(t^a> 5\), and let \(p\) be a prime number. Assume also that \(\Delta(G)\) is connected with cut-vertex \(p\), and that both the sets of primes \(\pi_{-}=\pi(t^a-1)-\{2\}\) and \(\pi_{+}=\pi(t^a+1)-\{2\}\) are non-empty. Then, denoting by \(K\) the solvable residual of \(G\), one of the following conclusions holds.
\begin{enumeratei} 
\item \(K\) is isomorphic to \(\PSL{t^a}\) or to \(\SL{t^a}\); 
\item \(K\) contains a minimal normal subgroup \(L\) of \(G\) such that \(K/L\) is isomorphic to \(\SL{t^a}\) and \(L\) is the natural module for \(K/L\).
\item \(t^a=13\) and \(p=2\). \(K\) contains a minimal normal subgroup \(L\) of \(G\) such that \(K/L\) is isomorphic to \(\SL{13}\), and \(L\) is isomorphic to one of the two \(6\)-dimensional irreducible modules for \(\SL{13}\) over \(\F_3\).
\end{enumeratei}
\end{theorem}

\begin{proof}
Denoting by \(R\) be the solvable radical of \(G\), Theorem~\ref{0.2} yields that \(G/R\) is an almost-simple group with socle isomorphic to \(\PSL{t^a}\), and \(\V G=\pi(G/R)\cup\{p\}\). Note that Lemma~\ref{InfiniteCommutator} applies here, because \(t^a=9\) is excluded by our assumption of \(\pi_-\) being non-empty; so, either we get conclusion (a), or \(K\) has a non-trivial normal subgroup \(L\) such that \(K/L\) is isomorphic to \(\PSL{t^a}\) or to \(\SL{t^a}\), and every non-principal irreducible character of \(L/L'\) is not invariant in \(K\). Therefore, we can assume that the latter condition holds.

Consider then a non-principal \(\lambda\) in \(\irr{L/L'}\): as we said, the inertia subgroup \(I_K(\lambda)\) is a proper subgroup of \(K\) and the same clearly holds for \(I_K(\lambda)N\), where we set \(N/L=\zent{K/L}\). Thus \(I_K(\lambda)N/N\) is a suitable proper subgroup of \(K/N\cong\PSL{t^a}\). 

If \(t\) is not a divisor of \(|K:I_K(\lambda)|\), then \(I_K(\lambda)N/N\) is of type (ii) and it is easy to see that \(I_K(\lambda)/L\) contains a Sylow \(t\)-subgroup of \(K/L\) as a normal subgroup. Assuming for the moment that this happens for every non-principal \(\lambda\in\irr{L/L'}\), an application of Lemma~4 in \cite{Z} and Lemma~\ref{SL2Nq} yields that \(K/L\) is isomorphic to \(\SL{t^a}\) and \(L/L'\) is the natural module for \(K/L\). So, in order to get conclusion (b), we only have to show that \(L'\) is trivial (note that, once this is proved, \(L=\oh t K\) is a minimal normal subgroup of \(G\)), and this is what we do next.

Writing \(U\) for \(L'\) we observe that, if \(K/L\cong\SL{t^a}\) and \(L/U\) is the natural module for \(K/L\), then \(\Delta(K/U)\) has two complete connected components, \(\{t\}\) and \(\V {K/L}-\{t\}\) (see the last paragraph of Section~2); recalling also that every prime in \(\pi(G/R)-\pi(K/L)\) is adjacent in \(\Delta(G/R)\) to all the other primes in \(\pi(G/R)-\{t\}\) (Theorem~\ref{MoretoTiep}), we see that the subgraph of \(\Delta(G)\) induced by the set \(\pi(G/R)-\{t\}\) is a clique, therefore $p\neq t$ and the set of neighbours of \(t\) in \(\Delta(G)\) consists of \(p\) only. Given that, for a proof by contradiction, assume \(U\neq 1\) and consider the non-abelian factor group \(L/U'\); since \(U/U'\) is an abelian normal subgroup of \(L/U'\) having \(t\)-power index, any non-linear irreducible character \(\phi\) of \(L/U'\) has a degree divisible by \(t\). Now, if \(p\neq 2\), then the non-adjacency between \(t\) and \(2\) yields that every irreducible character of \(K\) lying over \(\phi\) has odd degree, which implies (via Theorem~A of \cite{NT}) that \(K/L\) has abelian Sylow \(2\)-subgroups. Since this is not the case, we deduce that \(p=2\) and so \(|K:I_K(\phi)|\) is divisible only by \(2\) and \(t\). 

Note that \(\phi\) cannot be invariant in \(K\), because the Schur multiplier of \(K/L\) is trivial and we would easily get a contradiction by Gallagher's theorem. It is also not difficult to see that \(I_K(\phi)N/N\) cannot be a subgroup of type (i), (ii) or (iv) of \(K/N\). On the other hand, \(I_K(\phi)N/N\) cannot be isomorphic to \(S_4\) or to \(A_4\) as well, because we know that every prime in \(\pi_-\cup\pi_+\) does not divide \(|K:I_K(\phi)N|\) and we would have \(\pi_-\cup\pi_+\subseteq\{3\}\), against the assumption that both \(\pi_-\) and \(\pi_+\) are non-empty. It remains to consider the case \(I_K(\phi)N/N\cong A_5\). In this situation, \(\pi_-\cup\pi_+\) is forced to coincide with \(\{3,5\}\), hence \(t\not\in\{3,5\}\). Moreover, \(\phi\) does not have an extension to \(I_K(\phi)\), as otherwise we would get (again by Gallagher's theorem) that \(t\) is adjacent to \(3\) and to \(5\) in \(\Delta(G)\). Now, if \(I_K(\phi)\) contains \(N\), then \(I_K(\phi)/L\) is isomorphic to \(\SL 5\) (because \(K/L\cong\SL {t^a}\) has a unique involution); but the Schur multiplier of this group is trivial, so \(\phi\) would have an extension to \(I_K(\phi)\). It follows that \(I_K(\phi)\) does not contain \(N\) and, using character triples, we see that there exists an irreducible character of \(I_K(\phi)\) lying over \(\phi\) whose degree is divisible by \(3\). This again produces the adjacency between \(t\) and \(3\) in \(\Delta(G)\), the final contradiction that yields conclusion (b) in this case.


\smallskip

So, our assumption will be henceforth that there exists a non-principal \(\lambda\in\irr{L/U}\) such that \(t\) divides \(|K:I_K(\lambda)|\). Our aim will be to get conclusion (c) under this hypothesis.

We claim that, if this is the setting, then the vertices \(2\) and \(t\) are adjacent in \(\Delta(K)\) (thus, in \(\Delta(G)\)). Assuming the contrary, first of all we observe that \(I_K(\lambda)/L\) is forced to contain a Sylow \(2\)-subgroup of \(K/L\), and \cite[Theorem~A]{NT} ensures that \(K/L\) has abelian Sylow \(2\)-subgroups, which yields \(N=L\); in particular, \(I_K(\lambda)/N\) cannot be a subgroup of type~(ii) of \(K/N\), and it cannot be isomorphic to \(S_4\) or to ${\rm PGL}_2(t^b)$ for any \(b\) (recall that the Sylow \(2\)-subgroups of ${\rm PGL}_2(t^b)$ are dihedral groups). The remaining possibilities for \(I_K(\lambda)/N\) are then to be either of type~(i) or isomorphic to a group in the following list: \(A_4\), \(A_5\), \(\PSL{t^b}\) for a suitable divisor $b$ of $a$. Suppose first that \(\lambda\) does not have an extension to \(I_K(\lambda)\): then, working in a Schur cover of \(I_K(\lambda)/N\) for any of its possible isomorphism types and using the theory of character triples, we see that there exists an irreducible character of \(I_K(\lambda)\) lying over \(\lambda\) and having an even degree. In fact, if \(I_K(\lambda)/N\) is (non-cyclic) of type~(i), then any central extension of \(I_K(\lambda)/N\) has an abelian normal subgroup of index \(2\) and all irreducible character degrees in \(\{1,2\}\); the other cases can be easily checked. This is against our assumption that \(2\) and \(t\) are non-adjacent in \(\Delta(K)\). On the other hand, if \(\lambda\) does have an extension to \(I_K(\lambda)\), then Gallagher's theorem yields a contradiction for all the possible types of \(I_K(\lambda)/N\), except for the type \(A_4\). In this last case, however, by Gallagher's theorem the primes in \(\pi(K/N)-\{2\}\) are the vertices of a clique in \(\Delta(K)\): taking into account Lemma~\ref{PSL2bis} and Theorem~\ref{MoretoTiep} (together with the fact that \(\V G=\pi(G/R)\cup\{p\}\)), it is easily seen that \(\Delta(G)\) cannot have a cut-vertex, against our hypothesis. The claim concerning the adjacency between \(2\) and \(t\) is then proved. As a consequence, again in view of Lemma~\ref{PSL2bis} and Theorem~\ref{MoretoTiep}, we get \(p=2\).

Still assuming that there exists a non-principal \(\lambda\in\irr{L/U}\) such that \(t\) divides \(|K:I_K(\lambda)|\), we also claim that \(t\) must be adjacent in \(\Delta(K)\) to some odd prime divisor of \(t^{2a}-1\). This is clearly true if \(I_K(\lambda)N/N\) is of type (i), (ii) or (iv). As regards type (iii), assume that our claim is false. Then \(I_K(\lambda)N/N\) cannot be isomorphic to \(A_4\) or \(S_4\) by our assumption that both \(\pi_+\) and \(\pi_-\) are non-empty; on the other hand, if \(I_K(\lambda)N/N\cong A_5\), then we get \(\pi_-\cup\pi_+=\{3,5\}\), and Gallagher's theorem or the theory of character triples (according to whether \(\lambda\) extends to \(I_K(\lambda)\) or not) yield the contradiction that \(t\) is adjacent to \(3\) in \(\Delta(K)\).

Finally, given our assumption that \(\Delta(G)\) has a cut-vertex (and still having \ref{PSL2bis}, \ref{MoretoTiep} in mind), it is easy to see that the neighbours of \(t\) in \(\Delta(G)\) belonging to \(\pi(t^{2a}-1)-\{2\}\) must be either all in \(\pi_-\) or all in \(\pi_+\); moreover, there are no adjacencies in \(\Delta(G)\) between vertices lying in \(\pi_-\) and vertices lying in \(\pi_+\). 

\smallskip
Taking into account the structure of the graph \(\Delta(G)\) as described so far, we consider now a chief factor \(V=L/Y\) of \(K\) and discuss the action of \(K/L\) on its irreducible module \(V\). We will work actually on the dual module of \(V\), analyzing the subgroups \(I_K(\mu)N/N\) for \(\mu\) in \(\irr V-\{1_V\}\); recall that, since \(\mu\) can be regarded as a character of \(L/U\), the factor group \(I_K(\mu)N/N\) is a proper subgroup of \(K/N\). 

\smallskip
Let us start by assuming that \(t\) is adjacent in \(\Delta(K)\) to some vertex in \(\pi_-\). 
Then we know that \(t\) has no neighbours in \(\pi_+\) (as a vertex of \(\Delta(G)\)), so we can immediately exclude that \(I_K(\mu)N/N\) is a subgroup of type (i$_-$); on the other hand, \(I_K(\mu)N/N\) can be a subgroup of type (i$_+$), and in that case it contains a Sylow subgroup of \(K/N\) (as a normal subgroup) for every prime in \(\pi_+\). However, note that the latter situation cannot occur \emph{for all} non-principal \(\mu\in\irr V\), as otherwise we get a contradiction by Lemma~\ref{SL2Nq}.  

Suppose that \(I_K(\mu)N/N\) is of type (ii): in this case, \(I_K(\mu)N/N\) must contain a Sylow subgroup of \(K/N\) both for the prime \(t\) and for all the primes in \(\pi_-\),  therefore it is a Frobenius group whose complements have order divisible by every prime in \(\pi_-\) and it has irreducible characters whose degree is a multiple of all those primes. It is not difficult to see that \(I_K(\mu)/L\) has a normal \(2\)-complement \(H/L\), which enjoys the same properties mentioned above for \(I_K(\mu)N/N\). If \(\mu\) extends to \(H\), then Clifford's theory yields the adjacency in \(\Delta(K)\) of every prime in \(\pi_-\) with every prime in \(\pi_+\), not our case. 
On the other hand, if \(\mu\) does not extend to \(H\), then \(V=L/Y\) is a \(t\)-group (recall \cite[8.16, 11.22, 11.31]{Is}) and we reach a contradiction as well: in fact, in this case \(I_K(\mu)/Y\) has a normal Sylow \(t\)-subgroup \(T_0/Y\) and \(\mu\) does not extend to \(T_0\). Now, the restriction of any \(\phi\in\irr{I_K(\mu)|\mu}\) to \(T_0\) must have an irreducible constituent whose degree is divisible by \(t\), as otherwise some linear constituent of \(\phi_{T_0}\) would lie over \(\mu\) and would be an extension to \(T_0\) of it, so Clifford's theory would imply the adjacency of \(t\) with every prime in \(\pi_+\).  

We exclude next that \(I_K(\mu)N/N\) is of type (iv). If this is the case, certainly \(|K:I_K(\mu)|\) is divisible by \(t\) (thus not divisible by any prime in \(\pi_+\)), and therefore no irreducible character of \(I_K(\mu)\) lying over \(\mu\) has a degree divisible by a prime in \(\pi_+\); it is not difficult to check that both if \(\mu\) extends to \(I_K(\mu)\) and if it does not, using Gallagher's theorem or character triples respectively, we obtain a contradiction (note that \(I_K(\mu)N/N\) cannot be isomorphic to \(\PSL{9}\) or \({\rm{PGL}}_2(9)\) in this case, because any primitive prime divisor of \(3^{2a}-1\) lies in \(\pi_+\) and clearly divides \(|K:I_K(\mu)|\)).

Our conclusion so far is that, for every non-principal \(\mu\in\irr V\), the factor group \(I_K(\mu)N/N\) is either of type (i$_+$) or of type (iii); as observed, there must be some \(\mu\) as above for which the latter case holds. Now, if \(I_K(\mu)N/N\) is isomorphic to either \(S_4\) or \(A_4\), then \(t\) cannot be \(3\): otherwise the \(t\)-part of \(|K/N|\) would be larger than \(t\), so \(t\) would divide \(|K:I_K(\mu)|\) and no prime in \(\pi_+\) could be a divisor of \(|K:I_K(\mu)N|\). But this would force \(\pi_+\) to be empty, which is not the case. Thus, \(t\) is a divisor of \(|K:I_K(\mu)|\) and \(\pi_+\) is then forced to be \(\{3\}\). If \(\mu\) extends to \(I_K(\mu)\), then Clifford's theory yields the adjacency in \(\Delta(K)\) between \(t\) and \(3\), which is not allowed. It only remains the possibility that \(\mu\) does not extend to \(I_K(\mu)\), and this can happen only if \(V\) is a \(2\)-group. On the other hand, if \(I_K(\mu)N/N\cong A_5\), a similar argument shows that either \(t\) is larger than \(5\) and \(\pi_+\subseteq\{3,5\}\), or \(\{t\}\cup\pi_+=\{3,5\}\); in any case \(t\) divides \(|K:I_K(\mu)|\) and, if \(\mu\) extends to \(I_K(\mu)\), then we get the adiacency of \(t\) with the primes in \(\pi_+\). This cannot happen, so again we have that \(V\) is a \(2\)-group. 

To sum up, we can assume that \(V\) is an irreducible \(K/L\)-module over the field \(\F_2\). Moreover, there exists \(r\in\{3,5\}\) (dividing \(t^a+1\)) such that every non-principal irreducible character of \(V\) is centralized by an element of order \(r\). An application of Lemma~\ref{brauer} yields the final contradiction for this case.

\medskip
Let us move to the case when \(t\) is adjacent in \(\Delta(K)\) to some vertex in \(\pi_+\). 
Then we know that \(t\) has no neighbours in \(\pi_-\) (as a vertex of \(\Delta(G)\)), so we can immediately exclude that \(I_K(\mu)N/N\) is of type~(i$_+$).

We can also exclude that \(I_K(\mu)N/N\) is of type (iv) arguing as in the case when \(t\) is adjacent to a prime in \(\pi_-\), unless \(I_K(\mu)N/N\) is isomorphic to \(\PSL{9}\) or \({\rm{PGL}}_2(9)\). In the latter case, it is not difficult to see that \(I_K(\mu)/L\) has a normal subgroup \(H/L\) isomorphic either to \(\PSL{9}\) or to \(\SL{9}\); recalling that \(\ker \mu\) is a normal subgroup of \(I_K(\mu)\) with \(L/\ker \mu\subseteq\zent{I_K(\mu)/\ker\mu}\), we consider the factor group \(H/\ker\mu\). If this factor group splits over \(L/\ker\mu\), then \(\mu\) extends to \(H\), and there certainly exists an irreducible character of \(I_K(\mu)\) lying over \(\mu\) with a degree divisible by \(5\); now, Gallagher's theorem yields the contradiction that \(t=3\) is adjacent to \(5\in\pi_-\). On the other hand, if \(H/\ker\mu\) does not split over \(L/\ker\mu\), then \(L/\ker\mu\) embeds in the Schur multiplier of \(H/L\) (so, \(|L/\ker\mu|\in\{2,3\}\)), and using character triples we see that \(H/\ker\mu\) has irreducible characters lying over \(\mu\) having a degree divisible by \(5\). As a consequence, \(I_K(\mu)\) has irreducibe characters lying over \(\mu\) having a degree divisible by \(5\), again producing the same contradiction as above.

Next, \(I_K(\mu)N/N\) can be of type (i$_-$), and in that case it contains a Sylow subgroup of \(K/N\) (as a normal subgroup) for every prime in \(\pi_-\). Clearly, \(I_K(\mu)/L\) contains a Sylow subgroup of \(K/L\) as a normal subgroup for every prime in \(\pi_-\) as well.

The factor group \(I_K(\mu)N/N\) can also be of type (ii) (with an order divisible by \(t\)): if so, recalling that there are no adjacencies between vertices in \(\pi_-\) and vertices in \(\pi_+\), then it must contain a Sylow subgroup of \(K/N\) for all the primes in \(\pi_-\) (of course \(I_K(\mu)/L\) contains a Sylow subgroup of \(K/L\) for every prime in \(\pi_-\) as well), and it is a Frobenius group whose kernel is a \(t\)-group and whose complements have an order divisible by every prime in \(\pi_-\). 
We claim that \(I_K(\mu)N/N\) must contain a full Sylow \(t\)-subgroup of \(K/N\).
To show this, we can clearly assume \(a>1\); if \(t^a-1\) has a primitive prime divisor \(q\), then \(q\) lies in \(\pi_-\) and so the Sylow \(t\)-subgroup of \(I_K(\mu)N/N\) (on which an element of order \(q\) of \(I_K(\mu)N/N\) acts fixed-point freely) cannot have order less than \(t^a\). On the other hand, assume that \(t^a-1\) does not have a primitive prime divisor and, for a proof by contradiction, that \(I_K(\mu)N/N\) does not contain a full Sylow \(t\)-subgroup of \(K/N\). Then \(a=2\), \(t\) is a Mersenne prime, and \(I_K(\mu)N/N\) has a (normal) subgroup \(H/N\) that is a Frobenius group whose kernel has order \(t\) and whose complements are cyclic of odd order \(\frac{t-1}{2}\). Setting \(H_0=H\cap I_K(\mu)\), we then get that the normal subgroup \(H_0/L\) of \(I_K(\mu)/L\) has irreducible characters of degree divisible by every prime in \(\pi_-\), and it has a trivial Schur multiplier because every Sylow subgroup of \(H_0/L\) is cyclic. Now \(\mu\) extends to \(H_0/L\), which easily yields that \(I_K(\mu)\) has irreducible characters lying over \(\mu\) and having a degree divisible by every prime in \(\pi_-\). As usual, we get a contradiction via Gallagher's theorem.

It remains to consider type (iii). Assume that \(\mu\in\irr V\) is such that \(I_K(\mu)N/N\) is isomorphic to either \(S_4\) or \(A_4\). Then \(t\) cannot be \(3\): otherwise the \(t\)-part of \(|K/N|\) would be larger than \(t\), so \(t\) would divide \(|K:I_K(\mu)|\) and no prime in \(\pi_-\) could be a divisor of \(|K:I_K(\mu)N|\). But this would force \(\pi_-\) to be empty, which is not the case. 
Thus, \(t\) is a divisor of \(|K:I_K(\mu)|\) and \(\pi_-\) is then forced to be \(\{3\}\). If \(\mu\) extends to \(I_K(\mu)\), then Clifford's theory yields the adjacency in \(\Delta(K)\) between \(t\) and \(3\), which is not allowed. It only remains the possibility that \(\mu\) does not extend to \(I_K(\mu)\), and this can happen only if \(V\) is a \(2\)-group. On the other hand, if \(I_K(\mu)N/N\cong A_5\), a similar argument shows that either \(t\) is larger than \(5\) and \(\pi_-\subseteq\{3,5\}\), or \(\{t\}\cup\pi_-=\{3,5\}\); in any case \(t\) divides \(|K:I_K(\mu)|\) and, if \(\mu\) extends to \(I_K(\mu)\), then we get the adiacency of \(t\) with the primes in \(\pi_-\). This cannot happen, so again we have that \(V\) is a \(2\)-group. 

To sum up, if there exists \(\mu\in\irr V\) such that \(I_K(\mu)N/N\) is of type (iii), then \(V\) is a non-trivial irreducible \(K/L\)-module over the field \(\mathbb{F}_2\) and the set \(\pi_-\) is contained in \(\{3,5\}\). But now, for \(r\in\pi_-\), the discussion in the last three paragraphs above ensures that the centralizer in \(K/L\) of every element of $V$ contains an element of order \(r\): an application of Lemma~\ref{brauer} yields a contradiction (note that \(t^a=7\) is excluded here by the fact that \(\pi_+\) is assumed to be non-empty). 

Our conclusion so far is that, for every non-trivial \(\mu\in\irr V\), the factor group \(I_K(\mu)/L\) either contains a unique Sylow subgroup of \(K/L\) for every prime in \(\pi_-\), or it contains a unique Sylow \(t\)-subgroup of \(K/L\). Recall, however, that \(V\) is not the natural module for \(K/L\) because the factor groups \(I_K(\mu)/L\) are never Sylow \(t\)-subgroups of \(K/L\) in the present situation. By (Lemma~\ref{SL2Nq} and) Theorem~\ref{TipoIeIIPieni} or Theorem~\ref{TipoIeII}, depending on whether \(V\) has order coprime to \(t\) or is a \(t\)-group, we then get that \(t^a=13\), \(K/L\) is isomorphic to \(\SL{13}\), and \(V\) is one of the two \(6\)-dimensional irreducible modules for \(\SL{13}\) over \(\F_3\).

\smallskip
In order to complete the proof of conclusion (c), it remains to show that actually \(|L|=3^6\) (which also implies \(L=\oh 3 K\trianglelefteq G\)). Observe first that, by our argument so far, \(|L/W|=3^6\) \emph{for every choice} of a chief factor \(L/W\) of \(K\); this implies that \(L/U\) is actually a \(3\)-group. If \(U\neq 1\), we see that every non-linear irreducible character \(\phi\) of \(L/U'\) has a degree divisible by \(3\). Now, such a \(\phi\) is not \(K\)-invariant, as otherwise it would have an extension to \(K\) (because the Schur multiplier of \(K/L\) is trivial) and we would get a contradiction via Gallagher's theorem; on the other hand, every maximal subgroup of \(K/N\cong\PSL{13}\) has an index in \(K/N\) that is divisible either by \(7\) or by \(13\), thus \(I_K(\phi)N\) cannot be a proper subgroup of \(K\) and we have a contradiction. Our conclusion so far is that \(U=1\), i.e. \(L\) is an abelian \(3\)-group.

Taking into accout that (again by \cite[8.16, 11.22, 11.31]{Is}) every irreducible character of \(L\) extends to its inertia subgroup in \(K\), we are now ready to finish our argument. Let \(\theta\) be a non-principal irreducible character of \(L\), and observe that \(\theta\) cannot be \(K\)-invariant (as the Schur multiplier of \(\SL{13}\) is trivial); therefore, \(I_K(\theta)N/N\) is contained in a maximal subgroup of \(K/N\). Now, the index of a maximal subgroup of \(\PSL{13}\) lies in the set \(\{14,78,91\}\), but the non-adjacency between \(3\) and \(13\) in \(\Delta(G)\) rules out the index \(78\) for a maximal subgroup containing \(I_K(\theta)N/N\). On the other hand, in all the remaining cases, \(3\) divides the order of \(I_K(\theta)N/N\); but now \(I_K(\theta)/L\) has an irreducible character of degree \(3\), yielding contradictory adjacencies via Gallagher's theorem (recall that \(3\) is also non-adjacent to \(7\) in \(\Delta(G)\)), except when it contains a Sylow \(3\)-subgroup of \(K/L\) as a normal subgroup. Since this holds for every non-principal \(\theta\in\irr L\), \cite[Lemma~4]{Z} yields that \(L\) is an irreducible \(K/L\)-module, completing the proof.
\end{proof}


\section{A proof of the main result}

In this section we prove the Main Theorem stated in the Introduction, which provides a characterization of the groups having a normal section isomorphic to \(\PSL {t^a}\) (for \(t\neq 2\) and \(t^a>5\)) and whose degree graph is connected with a cut-vertex. 

\begin{proof}[Proof of the Main Theorem] We start by proving the ``only if" part of the statement: we assume that the group \(G\) has a composition factor isomorphic to $\PSL{t^a}$ for a suitable odd prime \(t\) with \(t^a> 5\), and the graph \(\Delta(G)\) is connected with cut-vertex \(p\). 

As usual, this implies (via Theorem~\ref{0.2}) that \(G/R\) is an almost-simple group whose socle is isomorphic to \(\PSL{t^a}\), and \(\V G=\pi(G/R)\cup\{p\}\). Moreover, by Theorem~A of \cite{W1}, if \(t\) is a divisor of \(|G/KR|\) then both \(t\) and \(2\) turn out to be complete vertices of the graph \(\Delta(G/R)\), against our hypothesis on \(\Delta(G)\). Also, if \(p=t\), then we get \(\V G=\pi(G/R)\); but again \cite[Theorem~A]{W1} yields that the subgraph of \(\Delta(G/R)\) induced by \(\pi(G/R)-\{t\}\) is connected, not our case.

It remains to prove that one among (a), (b) and (c) holds and, in this respect, much of the work has been done in Theorem~\ref{MainAboutK}. Namely, we will only have to show that \(\V{G/K}=\{p\}\) in cases (a) and (b), whereas \(p=2\) and \(\V {G/K}\subseteq\{2\}\) in case (c). 

Set \(N=K\cap R\) (so, \(K/N\cong\PSL{t^a}\)): we start by proving that 
\(t\) cannot lie in \(\V{G/K}\). In fact, assuming the contrary and recalling that \(t\) does not divide \(|G/KR|\), we would have (by Ito-Michler's theorem) that \(t\) divides the degree of some irreducible character \(\phi\) of \(R/N\cong KR/K\). Now, as \(KR/N\) is the direct product of \(K/N\) and \(R/N\), it easily follows that \(t\) is adjacent in \(\Delta(G)\) to every prime in \(\pi(K/N)-\{t\}=\pi(KR/R)-\{t\}\). But an application of (Clifford's theory and) Proposition~2.6(a) in \cite{ACDPS} yields that \(t\) is adjacent in \(\Delta(G)\) to every prime \(r\) in \(\pi(G/R)-\pi(KR/R)\) as well: in fact, according to that proposition, there exists an irreducible character \(\theta\) of \(K/N\) such that \(|G:I_G(\theta)|\) is divisible by \(r\), and our claim follows considering any irreducible character of \(G\) lying over \(\phi\theta\in\irr{KR}\). So, \(t\) turns out to be a complete vertex in the subgraph of \(\Delta(G)\) induced by \(\pi(G/R)\); this forces \(p=t\), against what observed above.

Next, we show that any prime \(q\in\V{G/K}\) is adjacent in \(\Delta(G)\) to all the primes in \(\pi(G/R)-\{q\}\). Take then \(q\in\V{G/K}\) (recall that \(q\neq t\) by the previous paragraph): 
it is well known that the Steinberg character ${\sf St}$ of \(K/N\) has an extension to \(G\) (see for instance \cite{S}). Therefore, recalling that ${\sf St}(1)=t^a$, Gallagher's theorem yields that \(q\) is adjacent to \(t\) in \(\Delta(G)\). Now, if \(q\) divides \(|G/KR|\), then \(q\) is adjacent in \(\Delta(G/R)\) to every prime in \(\pi(G/R)-\{q,t\}\) by \cite[Theorem~A]{W1}; but since \(q\) is also adjacent to \(t\) in \(\Delta(G)\), we are done in this case. On the other hand, if \(q\) does not divide \(|G/KR|\), then \(q\) divides the degree of some irreducible character of \(KR/K\) (hence of \(R/N\)). Since \(KR/N=K/N\times R/N\), clearly \(q\) is adjacent in \(\Delta(G)\) to every prime in \(\pi(K/N)-\{q\}=\pi(KR/R)-\{q\}\); but $q$ is also adjacent in \(\Delta(G)\) to every prime in \(\pi(G/R)-\pi(KR/R)\), by the same argument involving  \cite[Proposition~2.6(a)]{ACDPS} that was used in the previous paragraph. The desired conclusion follows.

As a consequence of the paragraph above, if there exists a prime \(q\in\V{G/K}-\{p\}\), then $q$ is a complete vertex in the subgraph of \(\Delta(G)\) induced by \(\pi(G/R)\), which contradicts the fact that \(p\neq q\) is a cut-vertex of \(\Delta(G)\). In particular, we get  \(\V{G/K}\subseteq\{p\}\). Another immediate consequence is that, if (conversely) \(p\) lies in \(\V{G/K}\), then \(p\) is a complete vertex of \(\Delta(G)\).


For cases (a) and (b) we actually see that \(p\in\V{G/K}\). Assuming the contrary, we would have that \(G/K\) is abelian, which implies \(R/N\subseteq\zent{G/N}\). Now Theorem~\ref{LewisWhite} yields that \(\Delta(G)\) is disconnected, a contradiction (observe that the subgroup \(C\) of Theorem~\ref{LewisWhite} coincides with \(R\) in our situation). As for case (c), it can be checked via GAP \cite{GAP} that the set of irreducible character degrees of \(K\) is \(\{1,\;2\cdot 3,\;7,\,2^2\cdot 3,\;13,\;2\cdot 7,\;2^3\cdot 7\cdot 13\}\), therefore \(\pi(G/R)=\{2,3,7,13\}\) induces a connected subgraph of \(\Delta(G)\) and we get \(\V G=\pi(G/R)\); as \(2\) is a complete vertex of \(\Delta(G)\), we necessarily have \(p=2\) and the ``only if" part of the statement is proved. Observe that,  still in case (c), \(t=13\) is adjacent in \(\Delta(G)\) to the cut-vertex \(2\) but also to \(7\), so the exception pointed out in the last sentence of the statement is a genuine one. Note also that, by the last observation of the previous paragraph, \(p\) is a complete vertex of \(\Delta(G)\) in cases (a) and (b) because it lies in \(\V{G/K}\), but \(p=2\) is a complete vertex in case (c) as well, thus proving the first claim in the sentence that concludes the statement.

\medskip
We move now to the ``if" part. Observe first that, by Theorem~\ref{LewisWhite} and Theorem~4.1 of \cite{LW}, \(\Delta(G)\) is a connected graph under our hypotheses. Setting \(N=K\cap R\) as above, consider the Steinberg character ${\sf St}$ of \(K/N\); we already observed that ${\sf St}$, viewed by inflation as an irreducible character of \(K\) having degree \(t^a\), has an extension to \(G\). Thus, by Gallagher's theorem, \(t^a\cdot \psi(1)\) is the degree of an irreducible character of \(G\) for every \(\psi\in\irr{G/K}\). Since one of our assumptions in cases (a) and (b) is that \(p\) divides the degree of some irreducible character of \(G/K\), we get the adjacency of \(t\) with \(p\) in \(\Delta(G)\). Furthermore, if \(K\) is as in (c) then, as noted in the paragraph above, \(2\) is a complete vertex of \(\Delta(G)\), thus it is the only possible cut-vertex of \(\Delta(G)\) and it also adjacent to the vertex \(3\) in \(\Delta(G)\).

In order to conclude the proof, we will show that there exists \(v\in\V G\) which is adjacent \emph{only} to \(p\) in \(\Delta(G)\): in particular, \(v\) is the vertex \(t\) in cases (a) and (b) (which settles also the last claim of the statement), whereas it is the vertex \(3\) in case (c). To this end we will first prove that, for every \(\chi\in\irr G\) such that \(v\in\pi(\chi(1))\) and \(\ker\chi\supseteq L\), we have \(\pi(\chi(1))\subseteq\{v,p\}\) (where \(L\) is set to be the trivial group in case (a)); secondly, we will see that \(v\) does not divide \(\chi(1)\) for every \(\chi\in\irr G\) with \(\ker\chi\not\supseteq L\). 

So, let us start with \(\chi\in\irr G\) such that \(v\in\pi(\chi(1))\) and \(\ker\chi\supseteq L\). Consider first cases (a) and (b) (so, we set \(v=t\)). Since \(t\) does not divide \(|G/KR|\), the degree of an irreducible constituent \(\xi\) of \(\chi_{KR}\) is necessarily divisible by \(t\); moreover, taking into account that \(KR/L\) is a central product of \(K/L\) and \(R/L\), we have that \(KR/L\) is isomorphic to a quotient of the direct product \(K/L\times R/L\), hence \(\xi\) can be viewed as the product of two suitable irreducible characters \(\alpha\), \(\beta\) of \(K/L\) and \(R/L\), respectively. Since \(\V{G/K}=\{p\}\) and \(p\neq t\), it easily follows that \(t\) does not divide any irreducible character degree of \(R/L\), so \(t\) necessarily divides \(\alpha(1)\). But \(K/L\) has a unique irreducible character of degree divisible by \(t\) (see \cite[Theorem~2.1(i)]{GGLMNT}), that is ${\sf St}$: hence \(\alpha={\sf St}\) and, as \(\alpha\) has an extension to \(G\), \(\chi(1)\) is the product of \(\alpha(1)=t^a\) with the degree of some irreducible character of \(G/K\). We conclude that \(\pi(\chi(1))\subseteq\{p,t\}\), as wanted. As regards case (c) (for which we set \(v=3\)), we have \(G=KR\) and \(\chi\) can be thought as an irreducible character of \(K/L\times R/L\), as above; since \(3\) is only adjacent to \(p=2\) in \(\Delta(K/L)\), and \(\V{R/L}\subseteq\{2\}\), we easily deduce that \(\pi(\chi(1))\subseteq\{3,p\}\) in this case as well.

Finally, let \(\chi\in\irr G\) be such that \(\chi_L\) has a non-principal irreducible constituent \(\lambda\) (obviously case~(a) is not involved here). 
 Denoting by \(V_0\) a Sylow \(v\)-subgroup of \(R\), the hypothesis \(\V{G/K}\subseteq\{p\}\) implies that the factor group \(V_0N/N\) is an abelian normal Sylow \(v\)-subgroup of \(R/N\cong KR/K\). It is then easily seen that \(V_0/L\) is an abelian normal Sylow \(v\)-subgroup of \(R/L\) (in particular, \(V_0\trianglelefteq G\)), and recall also that, by Lemma~\ref{singer}, we have \(L\subseteq\zent {V_0}\). Now, consider the normal subgroup \(KV_0\) of \(G\): taking into account that, both in case (b) and in case (c), \(I_K(\lambda)/L\) is a Sylow \(v\)-subgroup of \(K/L\), we have that \(I_{KV_0}(\lambda)=I_K(\lambda)V_0\) is a Sylow \(v\)-subgroup of \(G\). Furthermore, again in both cases (b) and (c), \(N/L\) acts fixed-point freely on \(L\); thus the abelian subgroup \(L\) has a complement in \(G\) by the proposition in the Introduction of \cite{Cu}. It follows that \(\lambda\) extends to \(I_{KV_0}(\lambda)\) and, as \(I_{KV_0}(\lambda)/L\cong (I_K(\lambda)/L)\times (V_0/L)\) is abelian, every irreducible constituent of \(\lambda^{I_{KV_0}(\lambda)}\) is linear. Now Clifford's theory (together with the fact that \(v\) does not divide \(|KV_0:I_{KV_0}(\lambda)|\)) yields that \(v\) does not divide the degree of any irreducible character of  \(KV_0\) lying over \(\lambda\). But then \(v\) does not divide \(\psi(1)\), where \(\psi\) is an irreducible constituent of \(\chi_{KV_0}\) lying over \(\lambda\). As \(v\) does not divide \(|G:KV_0|\) as well, it follows that \(v\) does not divide \(\chi(1)\), and the proof is complete.
\end{proof}

We conclude the paper with the following remark, that compares the structure of the groups appearing in the Main Theorem with the structure of the non-solvable groups whose degree graph has two connected components.


\begin{rem}\label{comparison}
Let \(G\) be a group satisfying the assumptions of the ``only if" part of the Main Theorem, i.e. \(G\) has a composition factor isomorphic to \(\PSL{t^a}\) for a suitable odd prime \(t\) (with $t^a>5$), and \(\Delta(G)\) is connected with cut-vertex \(p\). Then, by the Main Theorem, the structure of \(G\) turns out to be very similar to the structure of a non-solvable group \(G\) whose graph \(\Delta(G)\) has two connected components (see Theorem~\ref{LewisWhite}), with the only exception of case (c).

In fact, in cases (a) and (b), we know that there exist normal subgroups \(K\supseteq N\) such that \(K/N\cong\PSL{t^a}\); these subgroups are respectively the last term in the derived series of \(G\), and \(N=K\cap R\). Also, observing that the subgroup denoted by \(C\) in the statement of Theorem~\ref{LewisWhite} is in fact \(R\), we have that \(t\) does not divide \(|G/KR|\). Finally, if \(N\neq 1\), then either \(K\cong\SL{t^a}\) or there exists a minimal normal subgroup \(L\) of \(G\) such that \(K/L\cong\SL{t^a}\) and \(L\) is isomorphic to the natural module for \(K/L\).

In other words, \(G\) has precisely the structure prescribed by Theorem~\ref{LewisWhite} except for two aspects: rather than being empty, the set \(\V{G/K}\) consists of a single prime which is the cut-vertex \(p\), and the vertex set of \(G\) is \(\pi(G/R)\cup\{p\}\) rather than \(\pi(G/R)\).
This can be somewhat surprising, since the graph-theoretical condition expressed in Theorem~\ref{LewisWhite} is in principle much stronger than that of the Main Theorem.

\smallskip
Of course, the two relevant classes of groups behave similarly also from the point of view of the graphs: the only difference is that the groups of the Main Theorem have a degree graph with a complete vertex \(p\) (that may be already in \(\pi(G/R)\) or not), which is the unique neighbour of \(t\) and which guarantees the connectedness of the graph. For a description of the relevant graphs, we refer the reader to \cite[Section~2]{DPSS}.

On the other hand, the class of groups as in case (c) of the Main Theorem is different and doesn't show up in Theorem~\ref{LewisWhite}. As already pointed out, this is the only exception to the fact that \(p\) (which is \(2\)) is the unique neighbour of \(t=13\) in \(\Delta(G)\) (see Introduction, Figure~\ref{c}). 
\end{rem}



\end{document}